\pdfoutput=1
\RequirePackage{ifpdf}
\ifpdf 
\documentclass[pdftex]{sigma}
\else
\documentclass{sigma}
\fi

\usepackage{bm}

\newcommand{\mlt}[1]{ \begin{tabular}[c]{@{}c@{}}#1\end{tabular}}
\newcommand\TT{\rule{0pt}{2.9ex}}
\newcommand\TB{\rule[-1.5ex]{0pt}{0pt}}
\newcommand\TTTt{\rule{0pt}{2.50ex}}
\newcommand\TTt{\rule{0pt}{2.60ex}}
\newcommand\TTb{\rule[-1.7ex]{0pt}{0pt}}
\newcommand\TTT{\rule{0pt}{3.1ex}}
\newcommand\TTB{\rule[-2.3ex]{0pt}{0pt}}

\newcommand{\R}{\mathbb{R}}
\newcommand{\C}{\mathbb{C}}
\newcommand{\Z}{\mathbb{Z}}
\newcommand{\N}{\mathbb{N}}
\newcommand{\expe}{{\rm e}}
\newcommand\Ga{\Gamma}

\newcommand{\F}[3]{{}_2F_1\hspace{-0.05cm}\left(
 \genfrac{}{}{0pt}{}{#1}{#2};#3\right)}
\newcommand{\FF}{\mathbf{F}}
\renewcommand{\P}{{\sf P}}
\newcommand{\Q}{\mathsf{Q}}
\newcommand{\QQ}{\boldsymbol{Q}}
\renewcommand{\L}[2]{L^{#1}_{#2}}

\makeatletter
\def\eqnarray{\stepcounter{equation}\let\@currentlabel=\theequation
\global\@eqnswtrue
\tabskip\@centering\let\\=\@eqncr
$$\halign to \displaywidth\bgroup\hfil\global\@eqcnt\z@
 $\displaystyle\tabskip\z@{##}$&\global\@eqcnt\@ne
 \hfil$\displaystyle{{}##{}}$\hfil
 &\global\@eqcnt\tw@ $\displaystyle{##}$\hfil
 \tabskip\@centering&\llap{##}\tabskip\z@\cr}

\def\endeqnarray{\@@eqncr\egroup
 \global\advance\c@equation\m@ne$$\global\@ignoretrue}

\def\@yeqncr{\@ifnextchar [{\@xeqncr}{\@xeqncr[5pt]}}
\makeatother

\numberwithin{equation}{section}

\newtheorem{thm}{Theorem}[section]

\newtheorem{Lemma}[thm]{Lemma}
 { \theoremstyle{definition}
\newtheorem{rem}[thm]{Remark}}

\begin{document}

\newcommand{\arXivNumber}{2009.07318}

\renewcommand{\PaperNumber}{053}

\FirstPageHeading

\ShortArticleName{Gauss Hypergeometric Representations of the Ferrers Function of the Second Kind}

\ArticleName{Gauss Hypergeometric Representations\\ of the Ferrers Function of the Second Kind}

\Author{Howard S.~COHL~$^{\rm a}$, Justin PARK~$^{\rm b}$ and Hans VOLKMER~$^{\rm c}$}

\AuthorNameForHeading{H.S.~Cohl, J.~Park and H.~Volkmer}

\Address{$^{\rm a)}$~Applied and Computational Mathematics Division,\\
\hphantom{$^{\rm a)}$}~National Institute of Standards and Technology,
Mission Viejo, CA 92694, USA}
\EmailD{\href{mailto:howard.cohl@nist.gov}{howard.cohl@nist.gov}}
\URLaddressD{\url{https://www.nist.gov/people/howard-cohl}}

\Address{$^{\rm b)}$~Department of Mathematics, Massachusetts Institute of Technology,\\
\hphantom{$^{\rm b)}$}~Cambridge, MA 02139, USA}
\EmailD{\href{mailto:justin.s.park77@gmail.com}{jpark00@mit.edu}}

\Address{$^{\rm c)}$~Department of Mathematical Sciences, University of Wisconsin-Milwaukee,\\
\hphantom{$^{\rm c)}$}~Milwaukee, WI 53211, USA}
\EmailD{\href{mailto:volkmer@uwm.edu}{volkmer@uwm.edu}}

\ArticleDates{Received September 17, 2020, in final form May 04, 2021; Published online May 20, 2021}

\Abstract{We derive all eighteen Gauss hypergeometric representations for the Ferrers function of the second kind, each with a different argument. They are obtained from the eighteen hypergeometric representations of the associated Legendre function of the second kind by using a limit representation. For the 18 hypergeometric arguments which correspond to these representations, we give geometrical descriptions of the corresponding convergence regions in the complex plane. In~addition, we consider a corresponding single sum Fourier expansion for the Ferrers function of the second kind. In~four of the eighteen cases, the determination of the Ferrers function of the second kind requires the evaluation of the hyper\-geo\-met\-ric function separately above and below the branch cut at $[1,\infty)$. In~order to complete these derivations, we use well-known results to derive expressions for the hypergeometric function above and below its branch cut. Finally we give a detailed review of the 1888 paper by Richard Olbricht who was the first to study hypergeometric representations of~Legendre functions.}


\Classification{33C05; 33C55; 42B05}

\section{Introduction}\label{sec1}
In 1888, Richard Emil Olbricht, a student of Felix Klein,
gave a list of 72 solutions of the associated Legendre differential equation~\cite[Section~3]{Olbricht1888}
(see Appendix~\ref{Olbrichtseventytwo} for a detailed description of
Olbricht's analysis),~\cite[equation~(14.2.1)]{NIST:DLMF}
\begin{gather}\label{Legendre}
\big(1-x^2\big)\frac{{\rm d}^2y}{{\rm d}x^2}-2x \frac{{\rm d}y}{{\rm d}x} +\left(\nu(\nu+1)-\frac{\mu^2}{1-x^2}\right) y=0,
\end{gather}
in terms of the Gauss hypergeometric series~\cite[Chapter 15]{NIST:DLMF}.
For instance, the first entry of the list is
\begin{gather*}
\left(\frac{1-x}{2}\right)^{\frac12\mu}\left(\frac{1+x}{2}\right)^{\frac12\mu} \F{\mu-\nu,\nu+\mu+1}{1+\mu}{\frac{1-x}{2}}\!.
\end{gather*}
In fact, as Olbricht has
pointed out, there are 18 possible
arguments $w$ of this type of
hyper\-geo\-metric series. Olbricht
has arranged these arguments into three separate groups
which we now~list.

Group I:
\begin{gather*}
w_1=\frac{1-x}{2},\quad\ w_2=\frac{1+x}{2},\quad\ w_3=\frac{x-1}{x+1},\quad\
 w_4=\frac{x+1}{x-1},\quad\ w_5=\frac{2}{1+x},\quad\ w_6=\frac{2}{1\!-x}.
\end{gather*}

Group II:
\begin{gather*}
w_7=1-x^2,\quad\ w_8=\frac{1}{1-x^2},\quad\ w_9=x^2,\quad\
w_{10}=\frac{1}{x^2},\quad\ w_{11}=\frac{x^2-1}{x^2},\quad\ w_{12}=\frac{x^2}{x^2-1}.
\end{gather*}

Group III:
\begin{gather*}
w_{13}=\frac{-x+\sqrt{x^2-1}}{2\sqrt{x^2-1}},\qquad w_{14}=\frac{x-\sqrt{x^2-1}}{x+\sqrt{x^2-1}},\qquad
w_{15}=\frac{2\sqrt{x^2-1}}{x+\sqrt{x^2-1}},
\\
w_{16}=\frac{2\sqrt{x^2-1}}{-x+\sqrt{x^2-1}},\qquad
w_{17}= \frac{x+\sqrt{x^2-1}}{2\sqrt{x^2-1}},\qquad w_{18}=\frac{x+\sqrt{x^2-1}}{x-\sqrt{x^2-1}}.
\end{gather*}
The associated Legendre functions of the first kind $P_\nu^\mu(x)$ and of the second kind $Q_\nu^\mu(x)$
are solutions of \eqref{Legendre} with simple behavior at the regular singularities $x=1$ and $x=\infty$, respectively. In~Magnus, Oberhettinger and Soni~\cite[pp.\ 155--163]{MOS} (see also~\cite[pp.~124--139]{Erdelyi}) these associated Legendre functions are expressed
as linear combinations of one or two of Olbricht's 72 solutions for every of the 18 possible $w_j$.
The hypergeometric representation{s} of the associated Legendre function $P_\nu^\mu(x)$ involving $w_j$, $1\le j\le 18$, is
labelled by the number $j$, and the corresponding formula{s} for $Q_\nu^\mu(x)$ {are} labelled by $j+18$.

The purpose of this paper is to generate a complete list of 18 such hypergeometric representations for the
Ferrers function of the second kind $\Q_\nu^\mu(x)$ which is also a solution of \eqref{Legendre}.
Some of these 18 representations can be found in the literature but we have not seen a complete list.
We {will} also discuss a related Fourier expansion and some {relevant} geometric topics. We pay special attention to
a precise formulation of the domains of validity of the stated representations in the complex plane including
a discussion of branch cuts.

Of course, the theory of the associated Legendre differential equation is a very classical topic. For instance, we
refer to Heine~\cite{Heine1878,Heine}, Hobson~\cite{Hob}, Lebedev~\cite{Lebedev}, Sch\"afke~\cite{Schafke63},
Magnus, Oberhettinger and Soni~\cite{MOS}, Andrews, Askey and Roy~\cite{AAR}, as well as Zhurina
and Karmazina~\cite{ZhurinaKarmazina1966}. In~such a well-investigated area one cannot claim to have new results.
This is a~review paper that, as we may hope, throws a new light on the theory of the Ferrers function of the second kind.
Since this review contains a large number
of individual formulas, we have summarized the
main results in Table~\ref{Tabsum}.
We will use the complex domains
\begin{gather}
\label{D1}
D_1:=\C\setminus((-\infty,-1]\cup[ 1,\infty)),\\
\label{D1p}
D_1^+:=D_1\cap\{x\colon \mathop{\rm Re} x>0\},\\
\label{D2}
D_2:=\C\setminus(-\infty,1],\\
\label{D2p}
D_2^+:=D_2\cap\{x\colon \mathop{\rm Re} x>0\},\\
\label{D3}
D_3:=\C\setminus(-1,\infty)=-D_2.
\end{gather}

\begin{rem}
The \href{http://dlmf.nist.gov}{DLMF} tabulates hypergeometric
representations for the Ferrers function of~the first kind~\cite[equations~(14.3.1) and~(14.3.11)]{NIST:DLMF} with arguments $(1-x)/2$ and $x^2$ respectively.
One may also find representations with
arguments $(x-1)/(x+1)$, $1-x^2$, $1-x^{-2}$ in~\cite[pp.~166--167]{MOS}. As far as we are
aware, an extensive list for all of the hypergeometric representations of~the Ferrers
function of the first kind does not exist in the literature. However, such a list could easily
be generated by starting with the hypergeometric representations of the associated Legendre
function of the first kind tabulated in~\cite{MOS} and using \eqref{FerrersP}.
\end{rem}

\begin{table}[h!]
\begin{center}
\caption{This table lists all Gauss hypergeometric representations
of the Ferrers function of the second kind presented in this paper. In~regard to Theorem~\ref{Fourier1}, $u=x+{\rm i}\sqrt{1-x^2}$, $v=x-{\rm i}\sqrt{1-x^2}$.}\vspace{1ex}
\label{Tabsum}\setlength{\tabcolsep}{5pt}
{\small\begin{tabular}{c|c|c|c|c|c|c} \hline
\mlt{ { Olbricht}\\[-.5ex] { group}} \TT\TB &
\mlt{ { Domain}\\[-.5ex] { of} $x$} & \mlt{{ Domain} \\[-.5ex] {\sc } { of} $\nu$} & \mlt{{ Domain}\\[-.5ex] { of} $\mu$} &\mlt{{ Extra}\\{ restriction}} & {Argument(s)} & {Theorem}
\\ \hline
\hline I &$x\in D_1$ &$\nu\in{\mathbb C}$ &$\mu\in{\mathbb C}\setminus{\mathbb Z}$ & $\nu+\mu\not\in-{\mathbb N}$& $\frac{1-x}{2}$ \TT\TB & \ref{tI1}
\\
\hline I &$x\in D_1$ &$\nu\in{\mathbb C}$ &$\mu\in{\mathbb C}\setminus{\mathbb Z}$ & $\nu+\mu\not\in-{\mathbb N}$& $\frac{1+x}{2}$ \TT\TB & \ref{tI2}
\\
\hline I &$x\in D_1$ &$\nu\in{\mathbb C}$ &$\mu\in{\mathbb C}\setminus{\mathbb Z}$ & $\nu+\mu\not\in-{\mathbb N}$& $\frac{x-1}{x+1}$ \TT\TB & \ref{tI3}
\\
\hline I &$x\in D_1$ &$\nu\in{\mathbb C}$ &$\mu\in{\mathbb C}\setminus{\mathbb Z}$& $\nu+\mu\not\in-{\mathbb N}$&$\frac{x+1}{x-1}$ \TT\TB& \ref{tI4}
\\
\hline I &$\pm\mathop{\rm Im} x\lessgtr 0$ &$\nu\in{\mathbb C},\ 2\nu\not\in{\mathbb Z}$ &$\mu\in{\mathbb C}$ & $\nu+\mu\not\in-{\mathbb N}_0$&$\frac{2}{1+x}$ \TT\TB& \ref{tI5}
\\
\hline I &$\pm\mathop{\rm Im} x\lessgtr 0$ &$\nu\in{\mathbb C},\ 2\nu\not\in{\mathbb Z}$ &$\mu\in{\mathbb C}$ & $\nu+\mu\not\in-{\mathbb N}_0$&$\frac{2}{1-x}$ \TT\TB& \ref{tI6}
\\
\hline I &$x\in D_1$ &$\nu\in{\mathbb C}$ &$\mu\in{\mathbb C}$ & $\nu+\mu\not\in-{\mathbb Z}$&$\frac{1-x}{2},\ \frac{1+x}{2}$ \TT\TB& \ref{tI7}
\\
\hline
\hline II &$x\in D_1^+$\TB &$\nu\in{\mathbb C}$ &$\mu\in{\mathbb C}\setminus{\mathbb Z}$ & $\nu+\mu\not\in
-{\mathbb N}$& $1-x^2$\TTTt& \ref{tII1}
\\
\hline II &$\pm\mathop{\rm Im} x\lessgtr 0$ &$\nu\in{\mathbb C},\ \nu+\frac12\not\in{\mathbb Z}$ &$\mu\in{\mathbb C}$ & $\nu+\mu\not\in-{\mathbb N}$& $\frac{1}{1-x^2}$\TT\TTb & \ref{tII2}
\\
\hline II &$x\in D_1$ &$\nu\in{\mathbb C}$ &$\mu\in{\mathbb C}$ & $\nu+\mu\not\in-{\mathbb N}$& $x^2$ & \ref{tII3}
\\
\hline II &$\pm\mathop{\rm Im} x\lessgtr 0$ &$\nu\in{\mathbb C},\ \nu+\frac12\not\in{\mathbb Z}$ &$\mu\in{\mathbb C}$ & $\nu+\mu\not\in-{\mathbb N}$&$\frac{1}{x^2}$ \TT\TTb & \ref{tII4}
\\
\hline II &$x\in D_1^+$ &$\nu\in{\mathbb C}$ &$\mu\in{\mathbb C}\setminus{\mathbb Z}$ & $\nu+\mu\not\in-{\mathbb N}$&$\frac{x^2-1}{x^2}$ \TT\TTb& \ref{tII5}
\\
\hline II &$x\in D_1$ &$\nu\in{\mathbb C}$ &$\mu\in{\mathbb C}$ & $\nu+\mu\not\in{\mathbb N}$&$\frac{x^2}{x^2-1}$ \TT\TTb& \ref{tII6}
\\
\hline
\hline III &$x\in D_1$ &$\nu\in{\mathbb C},\ \nu+\frac12\not\in{\mathbb Z}$ &$\mu\in{\mathbb C}$ & $\nu+\mu\not\in-{\mathbb N}$& $\frac{\pm x+{\rm i}\sqrt{1-x^2}}{2{\rm i}\sqrt{1-x^2}}$\TTT\TTB & \ref{tIII1}
\\
\hline III &$x\in D_1$ &$\nu\in{\mathbb C},\ \nu+\frac12\not\in{\mathbb Z}$ &$\mu\in{\mathbb C}$ & $\nu+\mu\not\in-{\mathbb N}$& $\frac{x\mp {\rm i}\sqrt{1-x^2}}{x\pm {\rm i}\sqrt{1-x^2}}$ \TTT\TTB& \ref{tIII2}
\\
\hline III &$x\in D_1^+$ &$\nu\in{\mathbb C}$ &$\mu\in{\mathbb C},\ 2\mu\not\in{\mathbb Z}$ & $\nu+\mu\not\in-{\mathbb N}$& $\frac{2{\rm i}\sqrt{1-x^2}}{\pm x+{\rm i}\sqrt{1-x^2}}$\TTB\TTT & \ref{tIII3}
\\
\hline
\hline III &$x\in D_1$ &$\nu\in{\mathbb C}$ &$\mu\in{\mathbb C}$ & $\nu+\mu\not\in-{\mathbb N}$&$\frac{u}{v},\ \frac{v}{u}$\TTt\TB & \ref{Fourier1}
\\
\hline
\end{tabular}}
\end{center}\vspace{-2ex}
\end{table}

This paper is organized as follows. In~Section~\ref{sec2} we provide definitions and properties of the special functions
used in this paper. In~Sections~\ref{sec3} (Group~I),~\ref{sec4} (Group~II), and~\ref{sec5} (Group~III), we~give the list of 18 Gauss hypergeometric representations of~$\Q_\nu^\mu(x)$
in terms of Olbricht's solutions and include some discussion of these results.
In~Section~\ref{sec6} we look at the Fourier expansion of~the function $(\sin \theta)^{-\mu} \Q_\nu^\mu(\cos\theta)$, and in
Section~\ref{sec7} we list the regions in the complex $x$-plane, where the hypergeometric arguments $|w_j|<1$ for each $j=1,2,\dots,18$.

\section{Special functions}\label{sec2}

We will use the following notations.
Let $\N := \{1, 2, 3, \ldots \}$, $\N_0:=\N\cup\{0\}$, $\Z:=\{0,\pm1,\pm2,\ldots\}$,
$\R$ real numbers, $\C$ complex numbers.

\subsection{{The gamma and Gauss hypergeometric functions}}
The gamma function $\Ga\colon\C\setminus-\N_0\to{\mathbb C}$
generalizes the factorial $n!=\Ga(n+1)$, $n\in\N_0$
(see~\cite[Chapter 5]{NIST:DLMF}).
We will use Euler's reflection
formula~\cite[equation~(5.5.3)]{NIST:DLMF}
\begin{gather}\label{Eulerreflect}
\Ga(z) \Ga(1-z) = \frac{\pi}{\sin(\pi z)}{,}
\end{gather}
and the duplication formula~\cite[equation~(5.5.5)]{NIST:DLMF}
\begin{gather*}
\Gamma(2z)=\frac{2^{2z-1}}{\sqrt{\pi}}\Ga(z)\Ga\big(z+\tfrac12\big).
\end{gather*}
The hypergeometric function
${}_2F_1(a,b;c;w)$, where
$(a,b,c,w)\in\C^2\times(\C\setminus-\N_0)\times
(\C\setminus[1,\infty))$
is a~complex valued analytic function which can be defined
in terms of the infinite series
for $|w|<1$,
\begin{gather}\label{Fseries}
 \F{a,b}{c}{w}:=\sum_{n=0}^\infty \frac{(a)_n(b)_n}{(c)_n} \frac{w^n}{n!},
\end{gather}
where $(a)_n=a(a+1)\cdots (a+n-1)$ is the Pochhammer symbol~\cite[equation~(5.2.4)]{NIST:DLMF},
and elsewhere by analytic continuation.
Frequently, we will use the following linear transformations
of the Gauss hypergeometric {function}, valid for
 $w\in \C\setminus[1,\infty)$, namely
Euler's and
and Pfaff's transformations
\cite[equation~(15.8.1)]{NIST:DLMF}
\begin{align}
\label{Euler}
\F{a,b}{c}{w}&=(1-w)^{c-a-b}\F{c-a,c-b}{c}{w}\\
\label{Pfaff1}
&=(1-w)^{-a}\F{a,c-b}{c}{\frac{w}{w-1}}\\
&=(1-w)^{-b}\F{b,c-a}{c}{\frac{w}{w-1}}\!. \label{Pfaff2}
\end{align}
Please see Appendix~\ref{hypercutsection} for a presentation of the values
of the hypergeometric function from above and below its branch cut $[1,\infty)$.
Unless stated otherwise the complex power $x^y$ denotes its principal value defined for $x\in \C\setminus(-\infty,0]$.

\subsection{The associated Legendre and Ferrers functions}

{We will follow the notation for Legendre functions as used in~\cite[Chapter 14]{NIST:DLMF}.}
The associated Legendre functions of the first kind $P_\nu^\mu(z)$ is given by~\cite[equation~(14.3.6)]{NIST:DLMF}
\begin{gather*}
P_\nu^\mu(z):=\frac{1}{\Ga(1-\mu)}\left(\frac{z+1}{z-1}\right)^{\frac12\mu}
\F{-\nu,\nu+1}{1-\mu}{\frac{1-z}{2}}\!.
\end{gather*}
It is defined for all $\nu,\mu\in \C$ and all $z\in D_2$.
The associated Legendre function of the second kind $Q_\nu^\mu({z})$ is given by~\cite[equation~(14.3.7)]{NIST:DLMF}
\begin{gather}\label{defQ}
Q_\nu^\mu(z):=\frac{\sqrt{\pi}\,\expe^{{\rm i}\pi\mu}\Ga(\nu+\mu+1)\big(z^2-1\big)^{\frac12\mu}}
{2^{\nu+1}\Ga\big(\nu+\frac32\big)z^{\nu+\mu+1}}
\F{\frac{\nu+\mu+2}{2},\frac{\nu+\mu+1}{2}}{\nu+\frac32}{\frac{1}{z^2}}\!,
\end{gather}
where $z\in D_2$ and $\nu+\mu\in\C\setminus -\N $.
Note that for any expression of the form $\big(z^2-1\big)^\alpha$, read this~as
\begin{gather*}
\big(z^2-1\big)^\alpha:=(z+1)^\alpha(z-1)^\alpha,
\end{gather*}
for any fixed $\alpha\in\C$ and $z\in D_2$.

The Ferrers functions of the first and second kind respectively
$\P_\nu^\mu(x)$, $\Q_\nu^\mu(x)$,
can be defined by~\cite[cf.~equations~(14.23.4) and~(14.23.5)]{NIST:DLMF}
\begin{gather}
\label{FerrersP}
\P_\nu^\mu (x) = \expe^{\frac12{\rm i}\pi\mu} P_\nu^\mu (x + {\rm i}0),
\\[.5ex]
\Q_\nu^\mu (x)=
\tfrac{1}{2} \big[\expe^{-\frac{1}{2}{\rm i}\pi\mu}
\big(\expe^{-{\rm i}\pi\mu} Q_\nu^\mu (x + {\rm i}0) \big)
+ \expe^{\frac{1}{2}{\rm i}\pi\mu} \big( \expe^{-{\rm i}\pi\mu} Q_\nu^\mu (x - {\rm i}0) \big)\big],
\label{FerrersQ}
\end{gather}
for $x\in(-1,1)$,
although the Ferrers functions can be extended analytically to the domain $D_1$.
Note that we have written \eqref{FerrersQ} in this
particular form since the table of associated Legendre
functions of the second kind found in
Magnus, Oberhettinger and Soni (1966)~\cite[pp.~155--163]{MOS}
(see also~\cite[pp.~124--139]{Erdelyi}), are listed in
terms of $\expe^{-{\rm i}\pi\mu}Q_\nu^\mu(z)$, so this form
of the limit representation
is particularly useful.

In regard to the name
\textit{Ferrers functions}
(also known as Legendre
functions on the cut), this
is the name adopted for these functions in the NIST DLMF~\cite[Chapter 14]{NIST:DLMF}.
This is because
the mathematician Norman Macleod
Ferrers (1829--1903), was the first
to extensively describe their properties~\cite{Ferrers1877}.
Of course, since they are intimately
connected to associated Legendre functions, another name as suggested by a referee could be Legendre--Ferrers functions.\footnote{One
may compute the Ferrers function of the first and second kind respectively
with the computer algebra system
{\sf Mathematica} by using the commands
\texttt{LegendreP[}$\nu$\texttt{,}$\mu$\texttt{,p,}$x$\texttt{]} and
\texttt{LegendreQ[}$\nu$\texttt{,}$\mu$\texttt{,p,}$x$\texttt{]}, where
$p=1,2$. In~the computer algebra system {\sf MAPLE} if you define \texttt{\_\,EnvLegendreCut\,:= -1..1;} (default), then you get the standard associated Legendre functions. Alternatively, if you define
\texttt{\_\,EnvLegendreCut\,:=1..infinity;} then you get the Ferrers functions.}\,\footnote{
The mention of specific products, trademarks, or brand names is for purposes of
identification only. Such mention is not to be interpreted in any way as an endorsement
or certification of such products or brands by the National Institute of Standards and
Technology, nor does it imply that the products so identified are necessarily the best
available for the purpose. All trademarks mentioned herein belong to their respective
owners.}

Common applications of the Ferrers functions include spherical or hyperspherical and~hyper\-spheroidal
harmonics (e.g., the quantum theory of angular momentum), conical functions, the~Mehler--Fock integral
transforms, opposite antipodal fundamental solutions of the Laplace--Beltrami operator and Helmholtz
operators on hyperspheres, and many other applications. The~Ferrers functions appear whenever one
performs harmonic analysis on spheres or hyperspheres with dimension greater than or equal to one because
they arise from the method of~sepa\-ra\-tion of
variables for the Laplace--Beltrami operator on these constant positive curvature
Riemannian manifolds. In~fact, the orthogonal Gegenbauer (or ultraspherical) polynomials provide a basis
for functions on hyperspheres (as well as on Euclidean space, hyperbolic geometry and other isotropic
manifolds) and are fundamentally connected to the Ferrers functions of the first kind (see~\cite[equation~(14.3.21)]{NIST:DLMF}.

The following relations expressing the Ferrers function of the second kind $\Q_\nu^\mu$
in terms of~the associated Legendre functions $P_\nu^\mu$ and $Q_\nu^\mu$ will be useful.
Equations \eqref{FerrersQa}
and \eqref{FerrersQd} can be
found in~\cite[equation~(14.23.6)]{NIST:DLMF}, and the other equations follow from
the connection relations~\cite[equations~(14.9.12) and~(14.9.15)]{NIST:DLMF}.

\begin{thm}\quad
\begin{enumerate}\itemsep=0pt
\item[$(a)$] For $\mathop{\rm Im} x>0$, we have
\begin{align}
\Q_\nu^\mu(x)&=\expe^{-\frac32\pi {\rm i}\mu} Q_\nu^\mu(x)+
\frac{\pi {\rm i}}2\expe^{\frac12\pi {\rm i}\mu} P_\nu^\mu(x)\label{FerrersQa}
\\
&= \frac12\pi \cot(\pi\mu)\expe^{\frac12\pi {\rm i}\mu}P_\nu^\mu(x)-\frac{\pi{\expe^{-\frac12\pi {\rm i}\mu}}}
{2\sin(\pi\mu)}
\frac{\Ga(\nu+\mu+1)}{\Ga(\nu-\mu+1)}P_\nu^{-\mu}(x)\label{FerrersQb}
\\
&= \expe^{-\frac12\pi {\rm i}\mu}\left(\cos(\pi\mu)-{\rm i}\frac{\sin(\pi(\mu-\nu))}{2\cos(\pi\nu)}\right)Q_\nu^\mu(x)\nonumber
\\
&\phantom{=}+{\rm i} \expe^{-\frac12\pi {\rm i}\mu}\frac{\sin(\pi(\mu-\nu))}{2\cos(\pi\nu)} Q_{-\nu-1}^\mu(x).\label{FerrersQc}
\end{align}

\item[$(b)$] For $\mathop{\rm Im} x<0$, we have
\begin{gather}\allowdisplaybreaks
\Q_\nu^\mu(x)=\expe^{-\frac12\pi {\rm i}\mu} Q_\nu^\mu(x)-\frac{\pi {\rm i}}2\expe^{-\frac12\pi {\rm i}\mu} P_\nu^\mu(x)\label{FerrersQd}
\\
\hphantom{\Q_\nu^\mu(x)}
{}= \frac12\pi \cot(\pi\mu)\expe^{-\frac12\pi {\rm i}\mu}P_\nu^\mu(x)
-\frac{\pi{\expe^{\frac12\pi {\rm i}\mu}}}{2\sin(\pi\mu)}
\frac{\Ga(\nu+\mu+1)}{\Ga(\nu-\mu+1)}P_\nu^{-\mu}(x)\label{FerrersQe}
\\ \hphantom{\Q_\nu^\mu(x)}
{}= \expe^{-\frac32\pi {\rm i}\mu}\left(\!\cos(\pi\mu)+{\rm i}\frac{\sin(\pi(\mu-\nu))}{2\cos(\pi\nu)}\right)Q_\nu^\mu(x)\nonumber
\\ \hphantom{\Q_\nu^\mu(x)=}
{}-{\rm i} \expe^{-\frac32\pi {\rm i}\mu}\frac{\sin(\pi(\mu-\nu))}{2\cos(\pi\nu)} Q_{-\nu-1}^\mu(x).\label{FerrersQf}
\end{gather}
\end{enumerate}
\end{thm}

\section[Group I hypergeometric representations for Q nu mu(x)]
{Group I hypergeometric representations for $\boldsymbol{\Q_\nu^\mu(x)}$}\label{sec3}

The following result appears in~\cite[p.~167]{MOS} in a slightly different form.
However, it is claimed therein that the result is only valid for $x\in(0,1)$,
where in fact it is valid for $x\in(-1,1)$. On the other hand, the same formula
is reproduced in terms of trigonometric functions in~\cite[p.~169]{MOS}, where
the full range $x\in(-1,1)$ is indicated.
This result is also stated in~\cite[equation~(14.3.2)]{NIST:DLMF}.

\begin{thm}\label{tI1}
Let $x\in D_1$, $\nu \in \C$, $\mu \in \C \setminus \Z$, $\nu+\mu\not\in-\N$. Then
\begin{gather*}
\Q_\nu^\mu (x) = \frac{\pi}{2 \sin(\pi\mu)} \Bigg[ \frac{\cos(\pi\mu)}{\Ga(1 - \mu)}
{\left( \frac{1+x}{1-x} \right)}^{{\frac12\mu}}
\F{-\nu, \nu+1}{1 - \mu}{\frac{1-x}{2}} \nonumber
\\ \hphantom{\Q_\nu^\mu (x)=\frac{\pi}{2 \sin(\pi\mu)} \Bigg[}
{}- \frac{\Ga(\nu+\mu+1)}{\Ga(\mu+1)\Ga(\nu-\mu+1)}
{\left( \frac{1-x}{1+x} \right)}^{{\frac12\mu}}
\F{-\nu, \nu+1}{1+\mu}{\frac{1-x}{2}}\! \Bigg].
\end{gather*}
\end{thm}
\begin{proof}
According to~\cite[equation~(14.3.20)]{NIST:DLMF}
and~\cite[Entry 19, p.~160]{MOS},
the associated Legendre function of the second kind satisfies
\begin{gather*}
\expe^{-{\rm i}\pi\mu} Q_\nu^\mu (z) =
\frac{\Ga(-\mu)\Ga(\nu+\mu+1)}{2 \Ga(\nu-\mu+1)}
\left(\frac{z-1}{z+1}\right)^{{\frac12\mu}}
\F{-\nu, \nu+1}{1+\mu}{\frac{1-z}{2}}\nonumber
\\ \hphantom{\expe^{-{\rm i}\pi\mu} Q_\nu^\mu (z) =}
{} + \frac12 \Ga(\mu)\left(\frac{z+1}{z-1}\right)^{{\frac12\mu}}
\F{-\nu, \nu+1}{1 - \mu}{\frac{1-z}{2}}\!,
\end{gather*}
where $z\in D_2$. This can be rewritten as
\begin{gather*}
\expe^{-{\rm i}\pi\mu} Q_\nu^\mu (z)
= \frac{-\Ga(\nu+\mu+1)}{2 \Ga(\nu-\mu+1) \sin(\pi\mu) \Ga(\mu+1)}
{\left( \frac{z-1}{z+1} \right)}^{{\frac12\mu}}
\F{-\nu, \nu+1}{1+\mu}{\frac{1-z}{2}}\nonumber
\\ \hphantom{\expe^{-{\rm i}\pi\mu} Q_\nu^\mu (z)= }
{}+ \frac{\pi}{2 \sin(\pi\mu)
\Ga(1 - \mu)} {\left( \frac{z+1}{z-1} \right)}^{{\frac12\mu}}
\F{-\nu, \nu+1}{1 - \mu}{\frac{1-z}{2}}\!,
\end{gather*}
\noindent using the reflection formula \eqref{Eulerreflect}.
Using \eqref{FerrersQ}, we derive from this equation an expression for~$\Q_\nu^\mu{(x)}$, namely
\begin{gather*}
\Q_\nu^\mu (x) = \frac{\pi}{4 \sin(\pi\mu)} \Bigg[ \frac{-{2}\Ga(\nu+\mu+1)}
{\Ga(\nu-\mu+1) \Ga(\mu+1)} {\left( \frac{1-x}{1+x} \right)}^{{\frac12\mu}}
\F{-\nu, \nu+1}{1+\mu}{\frac{1-x}{2}}
\\ \hphantom{\Q_\nu^\mu (x) = \frac{\pi}{4 \sin(\pi\mu)} \Bigg[}
{} + \frac{\expe^{-{\rm i}\pi\mu}+\expe^{{\rm i}\pi\mu}}
{\Ga(1 - \mu)} {\left( \frac{1+x}{1-x} \right)}^{{\frac12\mu}}
\F{-\nu, \nu+1}{1 - \mu}{\frac{1-x}{2}}\Bigg]
\\ \hphantom{\Q_\nu^\mu (x)}
{}= \frac{\pi}{2 \sin(\pi\mu)} \Bigg[ \frac{\cos(\pi\mu)}{\Ga(1 - \mu)}
{\left( \frac{1+x}{1-x} \right)}^{{\frac12\mu}}
\F{-\nu,\nu+1}{1 - \mu}{\frac{1-x}{2}}
\\ \hphantom{\Q_\nu^\mu (x)= \frac{\pi}{2 \sin(\pi\mu)} \Bigg[}
{} - \frac{\Ga(\nu+\mu+1)}{\Ga(\nu-\mu+1) \Ga(\mu+1)}
{\left( \frac{1-x}{1+x} \right)}^{{\frac12\mu}}
\F{-\nu, \nu+1}{1+\mu}{\frac{1-x}{2}}\Bigg],
\end{gather*}
which is true for all $x\in(-1,1)$. Since both sides of the claimed identity are analytic functions on $D_1$, the
identity is true for all $x\in D_1$.
\end{proof}

\begin{thm}\label{tI2}
Let $x \in D_1$, $\nu \in \C$, $\mu \in \C \setminus \Z$, such
that $\nu + \mu \notin -\N$. Then
\begin{gather*}
\Q_\nu^\mu (x) = -\frac12\Bigg[\cos(\pi\nu) \Ga(\mu)
{\left( \frac{1-x}{1+x} \right)}^{{\frac12\mu}}
\F{-\nu, \nu+1}{1 - \mu}{\frac{1+x}{2}}
\\ \hphantom{\Q_\nu^\mu (x) =}
{}+\cos(\pi(\nu + \mu))\Ga(-\mu)
\frac{\Ga(\nu+\mu+1)}{\Ga(\nu-\mu+1)}
{\left( \frac{1+x}{1-x} \right)}^{{\frac12\mu}}
\F{-\nu, \nu+1}{1+\mu}{\frac{1+x}{2}}\Bigg].
\end{gather*}
\end{thm}
\begin{proof}
According to~\cite[Entry 20, p.~160]{MOS}, the associated Legendre function
of the second kind satisfies
\begin{gather*}
\expe^{-{\rm i}\pi\mu} Q_\nu^\mu (z) = -\frac12 \expe^{\mp{\rm i}\pi\nu} \Ga(\mu)
{\left( \frac{z-1}{z+1} \right)}^{{\frac12\mu}}
\F{-\nu, \nu+1}{1 - \mu}{\frac{1 + z}{2}}\nonumber
\\ \hphantom{\expe^{-{\rm i}\pi\mu} Q_\nu^\mu (z) =}
{} - \frac12 \expe^{\mp{\rm i}\pi\nu} \frac{\Ga(\nu+\mu+1) \Ga(-\mu)}{\Ga(\nu-\mu+1)}
{\left( \frac{z+1}{z-1} \right)}^{{\frac12\mu}}
\F{-\nu, \nu+1}{1+\mu}{\frac{1 + z}{2}}\!,
\end{gather*}
when $\pm\mathop{\rm Im} z>0$.
Using \eqref{FerrersQ}, we derive from this equation an expression for $\Q_\nu^\mu{(x)}$, namely
\begin{gather*}
\Q_\nu^\mu (x) = -\frac{1}{4} \Bigg[ \expe^{-{\rm i}\pi\nu} \Ga(\mu)
{\left( \frac{1-x}{1+x} \right)}^{{\frac12\mu}}
\F{-\nu, \nu+1}{1 - \mu}{\frac{1 + x}{2}}
\\ \hphantom{\Q_\nu^\mu (x) = -\frac{1}{4} \Bigg[}
{} + {\expe^{-{\rm i}\pi(\nu+\mu)}} \frac{\Ga(\nu+\mu+1) \Ga(-\mu)}{\Ga(\nu-\mu+1)}
{\left( \frac{1+x}{1-x} \right)}^{{\frac12\mu}}
\F{-\nu, \nu+1}{1+\mu}{\frac{1 + x}{2}}
\\ \hphantom{\Q_\nu^\mu (x) = -\frac{1}{4} \Bigg[}
{} + \expe^{{\rm i}\pi\nu} \Ga(\mu)
{\left( \frac{1-x}{1+x} \right)}^{{\frac12\mu}}
\F{-\nu, \nu+1}{1 - \mu}{\frac{1 + x}{2}}
\\ \hphantom{\Q_\nu^\mu (x) = -\frac{1}{4} \Bigg[}
{}+{\expe^{{\rm i}\pi(\nu+\mu)}}
\frac{\Ga(\nu+\mu+1) \Ga(-\mu)}{\Ga(\nu-\mu+1)}
{\left( \frac{1+x}{1-x} \right)}^{{\frac12\mu}}
\F{-\nu, \nu+1}{1+\mu}{\frac{1 + x}{2}} \Bigg]
\\ \hphantom{\Q_\nu^\mu (x)}
{}= -\frac12 \Bigg[ \cos(\pi\nu) \Ga(\mu)
{\left( \frac{1-x}{1+x} \right)}^{{\frac12\mu}}
\F{-\nu, \nu+1}{1 - \mu}{\frac{1 + x}{2}}
\\ \hphantom{\Q_\nu^\mu (x)= -\frac12 \Bigg[}
{}+ \cos(\pi(\nu + \mu)) \frac{\Ga(\nu+\mu+1) \Ga(-\mu)}{\Ga(\nu-\mu+1)}
\left( \frac{1+x}{1-x} \right)^{\frac\mu2}
\F{-\nu, \nu+1}{1+\mu}{\frac{1 + x}2}
\Bigg].
\end{gather*}
This completes the proof.
\end{proof}

\begin{thm}\label{tI3}
Let $x \in D_1$, $\nu \in \C$, $\mu \in \C \setminus \Z$,
such that $\nu + \mu \notin -\N$. Then
\begin{gather*}
\Q_\nu^\mu (x) = \frac{(1 + x)^\nu}{2^{\nu + 1}} \Bigg[
\cos(\pi\mu) \Ga(\mu) {\left( \frac{1+x}{1-x} \right)}^{{\frac12\mu}}
\F{-\nu, -\nu - \mu}{1 - \mu}{\frac{x-1}{x+1}}\nonumber
\\ \hphantom{\Q_\nu^\mu (x) = \frac{(1 + x)^\nu}{2^{\nu + 1}} \Bigg[}
{}+ \frac{\Ga(\nu+\mu+1) \Ga(-\mu)}{\Ga(\nu-\mu+1)}
{\left( \frac{1 - x}{1 + x} \right)}^{{\frac12\mu}}
\F{-\nu, \mu - \nu}{1+\mu}{\frac{x-1}{x+1}} \Bigg].
\end{gather*}
\end{thm}
\begin{proof}
This follows from~\cite[Entry 21, p.~160]{MOS}.
\end{proof}

\begin{thm}\label{tI4}
Let $x \in D_1$, $\nu \in \C$, $\mu \in \C \setminus \Z$,
such that $\nu + \mu \notin -\N$. Then
\begin{gather*}
\Q_\nu^\mu(x)=\frac{-2^\nu}{(1-x)^{\nu+1}}\Bigg[
\Ga(\mu)\cos(\pi\nu)\left(\frac{1-x}{1+x}\right)^{{\frac12\mu}}
\F{\nu+1,\nu-\mu+1}{1-\mu}{\frac{x+1}{x-1}}\nonumber
\\ \hphantom{\Q_\nu^\mu(x)=}
{}+\frac{\Ga(-\mu)\cos(\pi(\nu+\mu))\Ga(\nu+\mu+1)}{\Ga(\nu-\mu+1)}
\left(\frac{1+x}{1-x}\right)^{{\frac12\mu}}\F{\nu+1,\nu+\mu+1}{1+\mu}{\frac{x+1}{x-1}}
\Bigg].
\end{gather*}
\end{thm}
\begin{proof}
This follows from~\cite[Entry 22, p.~160]{MOS}.
\end{proof}

\begin{thm}\label{tI5}
Let $\nu,\mu\in\C$ such that $\nu+\mu\notin -\N_0$, $2\nu\notin\Z$. If $\pm \mathop{\rm Im} x>0$ then
\begin{gather*}
\Q_\nu^\mu(x)= 2^\nu \left(\cos(\pi\mu)\mp {\rm i}\frac{\sin(\pi(\mu-\nu))}{2\cos(\pi\nu)}\right) \frac{\Ga(\nu+1)\Ga(\nu+\mu+1)}{\Ga(2\nu+2)}\nonumber
\\ \hphantom{\Q_\nu^\mu(x)=2^\nu}
{}\times (1+x)^{-1+\frac12\mu-\nu}(1-x)^{-\frac12\mu}\F{\nu-\mu+1,\nu+1}{2\nu+2}{\frac{2}{1+x}}\nonumber
\\ \hphantom{\Q_\nu^\mu(x)=}
{} \pm {\rm i}\pi 2^{-\nu-2}\sec(\pi\nu)\frac{\Ga(-\nu)}{\Ga(-2\nu)\Ga(\nu-\mu+1)}\nonumber
\\ \hphantom{\Q_\nu^\mu(x)=2^\nu}
{}\times (1+x)^{\nu+\frac12\mu}(1-x)^{-\frac12\mu}\F{-\nu,-\nu-\mu}{-2\nu}{\frac{2}{1+x}}\!.
\end{gather*}
\end{thm}
\begin{proof}
We represent $\Q_\nu^\mu(x)$ by \eqref{FerrersQc}
if $\mathop{\rm Im} x>0$ and by
\eqref{FerrersQf}
if $\mathop{\rm Im} x<0$. Then we use~\cite[Entry 23, p.~161]{MOS}, namely
\begin{gather*}
\expe^{-{\rm i}\pi\mu} Q_\nu^\mu (x) =
2^\nu\frac{\Ga(\nu+1) \Ga(\nu+\mu+1)}{\Ga(2\nu+2)}\frac{(x+1)^{{\frac12\mu}-\nu-1}}{(x-1)^{{\frac12\mu}}}
\F{\nu+1,\nu-\mu+1}{2\nu+2}{\frac{2}{1+x}}\!,
\end{gather*}
for $x\in D_2$
and the desired result follows.
\end{proof}

\begin{thm}\label{tI6}
Let $\nu,\mu\in\C$ such that $\nu+\mu\notin -\N_0$, $2\nu\notin\Z$. If $\pm \mathop{\rm Im} x>0$ then
\begin{gather*}
\Q_\nu^\mu(x)= 2^\nu \expe^{\mp\pi {\rm i}(\nu+1)}\left(\cos(\pi\mu)\mp {\rm i}\frac{\sin(\pi(\mu-\nu))} {2\cos(\pi\nu)}\right) \frac{\Ga(\nu+1)\Ga(\nu+\mu+1)}{\Ga(2\nu+2)}\nonumber
\\ \hphantom{\Q_\nu^\mu(x)= 2^\nu}
{}\times (1+x)^{\frac12\mu}(1-x)^{{-\nu-\frac12\mu-1}} \F{\nu+\mu+1,\nu+1}{2\nu+2}{\frac{2}{1-x}}\nonumber
\\ \hphantom{\Q_\nu^\mu(x)= }
 \pm {\rm i}\pi 2^{-\nu-2}\expe^{\pm \pi {\rm i}\nu}\sec(\pi\nu) \frac{\Ga(-\nu)}{\Ga(-2\nu)\Ga(\nu-\mu+1)}\nonumber
\\ \hphantom{\Q_\nu^\mu(x)= 2^\nu}
{}\times (1+x)^{\frac12\mu}(1-x)^{\nu-\frac12\mu}\F{-\nu,\mu-\nu}{-2\nu}{\frac{2}{1-x}}\!.
\end{gather*}
\end{thm}
\begin{proof}
We represent $\Q_\nu^\mu(x)$ by \eqref{FerrersQc}
if $\mathop{\rm Im} x>0$ and by \eqref{FerrersQf}
if $\mathop{\rm Im} x<0$. Then we use~\cite[Entry 24, p.~161]{MOS}, namely
\begin{gather*}
\expe^{-{\rm i}\pi\mu}
Q_\nu^\mu (x) = 2^\nu\frac{\Ga(\nu+1)\Ga(\nu+\mu+1)}{\Ga(2\nu+2)}
\frac{(x+1)^{\frac12\mu}}{(x-1)^{\nu+\frac12\mu+1}}
\F{\nu+1,\nu+\mu+1}{2\mu+2}{\frac{2}{1-x}}\!,
\end{gather*}
for $x\in D_2$
and the desired result follows.
\end{proof}

\begin{rem}
The hypergeometric representations of $\Q_\nu^\mu(x)$ given in the previous six theorems can be written
in slightly different forms by applying Euler's transformation~\eqref{Euler}. For instance, the result of Theorem~\ref{tI1} can be written as
\begin{gather*}
\Q_\nu^\mu (x) = \frac{\pi}{2 \sin(\pi\mu)} \Bigg[2^\mu \frac{\cos(\pi\mu)}{\Ga(1 - \mu)}
\big(1-x^2\big)^{-\frac\mu2}\F{\nu-\mu+1, -\mu-\nu}{1 - \mu}{\frac{1-x}{2}} \nonumber
\\ \hphantom{\Q_\nu^\mu (x) = \frac{\pi}{2 \sin(\pi\mu)} \Bigg[}
{}- 2^{-\mu}\frac{\Ga(\nu+\mu+1)}{\Ga(\mu+1)\Ga(\nu-\mu+1)}\big(1-x^2\big)^{\frac\mu2}
\F{\nu\!+\mu\!+1, \mu\!-\nu}{1+\mu}{\frac{1\!-x}{2}} \Bigg].
\end{gather*}
We observe that by applying Pfaff's transformation \eqref{Pfaff1} and~\eqref{Pfaff2},
the results of Theorems~\ref{tI3}, \ref{tI4} and~\ref{tI6}
follow immediately from those of Theo\-rems~\ref{tI1}, \ref{tI2} and~\ref{tI5}, respectively.
\end{rem}

Each of the previous six theorems represents $\Q_\nu^\mu(x)$ as a sum of two functions each of which is itself a solution of the associated Legendre equation \eqref{Legendre}. In~the following we identify these solutions.

\begin{rem}
In Theorems~\ref{tI1} and~\ref{tI3}, $\Q_\nu^\mu(x)$ is expressed as a linear combination of $\P_\nu^\mu(x)$ and $\P_\nu^{-\mu}(x)$.
This follows from the connection formula
\cite[cf.~equation~(14.9.2)]{NIST:DLMF}
\begin{gather}\label{connect1}
\frac2\pi \sin(\pi\mu)\Q_\nu^\mu(x) =\cos(\pi\mu)\P_\nu^\mu(x)-\frac{\Ga(\nu+\mu+1)}{\Ga(\nu-\mu+1)}\P_\nu^{-\mu}(x),
\end{gather}
and the hypergeometric representation~\cite[equation~(14.3.1)]{NIST:DLMF}
of the Ferrers function of the first kind
\begin{gather}\label{hyprep}
\P_\nu^\mu(x)=\left(\frac{1+x}{1-x}\right)^{\frac12\mu} \frac{1}{\Ga(1-\mu)}\F{-\nu,\nu+1}{1-\mu}{\frac{1-x}{2}}\!.
\end{gather}
Analogously, in Theorems~\ref{tI2} and~\ref{tI4}, $\Q_\nu^\mu(x)$ is written as a linear combination of $\P_\nu^\mu(-x)$
and~$\P_\nu^{-\mu}(-x)$. This follows from the connection relation
\begin{gather*}
\frac2\pi\sin(\pi\mu)\Q_\nu^\mu(x)=-\cos(\pi\nu)\P_\nu^\mu(-x)+\cos(\pi(\nu+\mu))
\frac{\Ga(\nu+\mu+1)}{\Ga(\nu-\mu+1)}\P_\nu^{-\mu}(-x){,}
\end{gather*}
which is a consequence of~\cite[equations~(14.9.7) and~(14.9.10)]{NIST:DLMF}. We {will use} this observation to derive additional hypergeometric representations of $\Q_\nu^\mu(x)$.
\end{rem}

The connection relation~\cite[equation~(14.9.10)]{NIST:DLMF}
\begin{gather*}
\frac2\pi\sin(\pi(\nu+\mu))\Q_\nu^\mu(x)=\cos(\pi(\nu+\mu))\P_\nu^\mu(x)-\P_\nu^\mu(-x),
\end{gather*}
and \eqref{hyprep}
lead to the following result.

\begin{thm}\label{tI7}
Let $x\in D_1$, $\mu,\nu\in \C$ such that $\nu+\mu\notin-\Z$. Then
\begin{gather*}
\Q_\nu^\mu(x)= \frac\pi2 \cot(\pi(\nu+\mu)) \left(\frac{1+x}{1-x}\right)^{\frac12\mu} \frac{1}{\Ga(1-\mu)} \F{-\nu,\nu+1}{1-\mu}{\frac{1-x}2}\nonumber
\\ \hphantom{\Q_\nu^\mu(x)=}
{}-\frac\pi2\csc(\pi(\nu+\mu)) \left(\frac{1-x}{1+x}\right)^{\frac12\mu}\frac{1}{\Ga(1-\mu)}
\F{-\nu,\nu+1}{1-\mu}{\frac{1+x}2}\!.
\end{gather*}
\end{thm}

\begin{rem}
The hypergeometric representation in Theorem~\ref{tI7} involves two hypergeometric functions with the same parameters
$a=\nu+1$, $b=-\nu$, $c=1-\mu$ but
two different arguments {$(1\pm x)/2$} while the previous results in this section involve two
hypergeometric functions with two different set of parameters but the same argument.
By representing $\Q_\nu^\mu(x)$ in terms of~$\big\{\P_\nu^\mu(-x), \P_\nu^{-\mu}(x)\big\}$, $\big\{\P_\nu^{-\mu}(-x), \P_\nu^\mu(x)\big\}$,
$\big\{\P_\nu^{-\mu}(-x), \P_\nu^{-\mu}(x)\big\}$, respectively, we can obtain three addi\-tional hypergeometric representations of
$\Q_\nu^\mu(x)$.
\end{rem}

\begin{rem}
In Theorem~\ref{tI5}, $\Q_\nu^\mu(x)$ is expressed as a linear combination of $Q_\nu^\mu(x)$ and $Q_{-\nu-1}^\mu(x)$ in the upper and lower half-plane.
The argument of the hypergeometric function is {$w_5=2/(1+x)$}.
If one considers $x\in(-1,1)$, then this function maps to $(1,\infty)$, where the hypergeometric
function takes two values depending on whether the value is approached from above or
below the ray $[1,\infty)$. These values can be computed by Theorems
\ref{Theorem1}--\ref{Theorem4} which are given in
Appendix~\ref{hypercutsection}.
We map $w_5(x\pm {\rm i}0)$ to
$(1+x)/2$, $(x-1)/(x+1)$, $(1-x)/2$, $(x+1)/(x-1)$, which have already
been encountered in
Theorems~\ref{tI2},~\ref{tI3},~\ref{tI1} and~\ref{tI4}, respectively.
Similar remarks apply to Theorem~\ref{tI6}.
\end{rem}

\section[Group II hypergeometric representations for Q nu mu (x)]
{Group II hypergeometric representations for $\boldsymbol{\Q_\nu^\mu(x)}$}\label{sec4}

\begin{thm}\label{tII1}
Let $x\in D_1^+$, $\nu\in\C$, $\mu\in\C\setminus\Z$, $\nu+\mu\in\C\setminus-\N$. Then
\begin{gather*}
\Q_\nu^\mu(x)=\frac{2^{\mu-1}\Ga(\mu)\cos(\pi\mu)}{\big(1-x^2\big)^{\frac12\mu}}
\F{\frac{\nu-\mu+1}{2},\frac{-\nu-\mu}{2}}{1-\mu}{1-x^2}
\\ \hphantom{\Q_\nu^\mu(x)=}
{}+\frac{\Ga(-\mu)\Ga(\nu+\mu+1)\big(1-x^2\big)^{\frac12\mu}}{2^{1+\mu}\Ga(\nu-\mu+1)}
\F{\frac{\nu+\mu+1}{2},\frac{\mu-\nu}{2}}{1+\mu}{1-x^2}\!.
\end{gather*}
\end{thm}

\begin{proof}
Use the Gauss hypergeometric representation of the associated Legendre function of the second kind~\cite[Entry 25, p.~161]{MOS},
\begin{gather*}
\expe^{-{\rm i}\pi\mu} Q_\nu^\mu (z) = 2^{\mu - 1} \Ga(\mu)(z^2 - 1)^{-{\frac12\mu}}
\F{\frac{\nu-\mu+1}{2}, \frac{-\nu-\mu}{2}}{1 - \mu}{1-z^2}
\\ \hphantom{\expe^{-{\rm i}\pi\mu} Q_\nu^\mu (z) =}
{} +\frac{\Ga(\nu+\mu+1) \Ga(-\mu)}{2^{1+\mu}\Ga(\nu-\mu+1)}(z^2 - 1)^{{\frac12\mu}}
\F{\frac{\nu+\mu+1}{2}, \frac{\mu-\nu}{2}}{1+\mu}{1 - z^2}\!,
\end{gather*}
valid for $z\in D_2$ with $\mathop{\rm Re} z>0$. Now \eqref{FerrersQ} gives the desired representation for $x\in(0,1)$.
Since both sides of the equation are analytic for $x\in D_1$, $\mathop{\rm Re} x>0${. The} full statement follows.
\end{proof}

\begin{thm}\label{tII2}
Let $\nu,\mu\in\C$ such that $\nu+\mu\notin-\N$ and $\nu+\frac12\notin\Z$. If $\pm \mathop{\rm Im} x>0$ then
\begin{gather*}
\Q_\nu^\mu(x)=\sqrt\pi\, 2^{-\nu-1} \expe^{\pm\frac12{\rm i}\pi(-\nu+\mu-1)}\left(\cos(\pi\mu)\mp {\rm i}\frac{\sin(\pi(\mu-\nu))}{2\cos(\pi\nu)}\right)\frac{\Ga(\nu+\mu+1)}{\Ga\big(\nu+\frac32\big)}
\\ \hphantom{\Q_\nu^\mu(x)=+}
{}\times \big(1-x^2\big)^{-\frac12\nu-\frac12} \F{\frac{\nu-\mu+1}2,\frac{\nu+\mu+1}2}{\nu+\frac32}{\frac{1}{1-x^2}}
\\ \hphantom{\Q_\nu^\mu(x)=}
{}+ \pi^{\frac32} 2^{\nu-1} \frac{\expe^{\pm\frac12{\rm i}\pi(\nu+\mu+1)}\sec(\pi\nu)}{\Ga(\nu-\mu+1)\Ga\big(\frac12-\nu\big)}
\big(1-x^2\big)^{\frac12\nu}\F{-\frac{\nu+\mu}2,\frac{\mu-\nu}2}{\frac12-\nu}{\frac{1}{1-x^2}}\!.
\end{gather*}
\end{thm}

\begin{proof}
We represent $\Q_\nu^\mu(x)$ by~\eqref{FerrersQc}
if $\mathop{\rm Im} x>0$ and by~\eqref{FerrersQf}
if $\mathop{\rm Im} x<0$. Then we use~\cite[Entry 26, p.~161]{MOS}
\begin{gather*}
\expe^{-{\rm i}\pi\mu}Q_\nu^\mu (z)=
\sqrt{\pi}\,2^{-\nu-1}\frac{\Ga(\nu+\mu+1)}{\Ga\big(\nu+\frac32\big)}(z^2-1)^{\tfrac12(\nu+1)}
\F{\frac{\nu-\mu+1}{2}, \frac{\nu+\mu+1}{2}}{\nu+\frac32}{\frac{1}{1-z^2}}\!,
\end{gather*}
for $z\in D_2$ and the desired result follows.
\end{proof}

\begin{thm}[equation~(14.3.12) in~\cite{NIST:DLMF}]\label{tII3}
Let $x \in D_1$ and $\nu, \mu \in \C$, such that $\nu + \mu \notin -\N$. Then
\begin{gather*}
\Q_\nu^\mu (x) = \frac{\sqrt{\pi}\, 2^{\mu - 1}}{(1 - x^2)^{\frac12\mu}} \Bigg[
\frac{-\sin\big(\frac{\pi}2(\nu+\mu)\big)
\Ga\big(\frac{\nu+\mu+1}{2}\big)}{\Ga\big(\frac{\nu-\mu+2}{2}\big)}
\F{-\frac{\nu+\mu}2,\frac{\nu-\mu+1}2}{\frac12}{x^2}
\\ \hphantom{\Q_\nu^\mu (x) = \frac{\sqrt{\pi}\, 2^{\mu - 1}}{(1 - x^2)^{\frac12\mu}} \Bigg[}
{} +\frac{2\cos\big(\tfrac{\pi}{2}(\nu+\mu)\big)
\Ga\big(\frac{\nu+\mu+2}{2}\big)\,x}{\Ga\big(\frac{\nu-\mu+1}{2}\big)}
\F{\frac{-\nu-\mu+1}{2}, \frac{\nu-\mu+2}{2}}{\frac32}{x^2}\! \Bigg].
\end{gather*}
\end{thm}

\begin{proof}
By~\cite[Entry 27, p.~161]{MOS},
the associated Legendre function of the second kind satisfies
\begin{gather*}
\expe^{-{\rm i}\pi\mu} Q_\nu^\mu (z) = \frac{\pi^\frac12 2^\mu \Ga\big(\frac{\nu+\mu+2}{2}\big) z}
{\Ga\big(\frac{\nu-\mu+1}{2}\big) (z^2 - 1)^{{\frac12\mu}}} \expe^{\pm \frac12 {\rm i} \pi (\mu - \nu)}
\F{\frac{-\nu-\mu+1}{2}, \frac{\nu-\mu+2}{2}}{\frac32}{z^2}
\\ \hphantom{\expe^{-{\rm i}\pi\mu} Q_\nu^\mu (z) =}
{}+ \frac{\pi^\frac12 2^{\mu - 1} \Ga\big(\frac{\nu+\mu+1}{2}\big)}{\Ga\big(\frac{\nu-\mu+1}{2}\big)
(z^2 - 1)^{{\frac12\mu}}} \expe^{\pm \frac12 {\rm i} \pi (\mu - \nu - 1)}
\F{\frac{-\nu-\mu}{2}, \frac{\nu-\mu+1}{2}}{\frac12}{z^2}\!,
\end{gather*}
where the upper and lower sign holds according to $\pm\mathop{\rm Im} z>0$.
Using this equation with \eqref{FerrersQ}, one derives
\begin{gather*}
\Q_\nu^\mu (x) = \frac{\pi^\frac12 2^{\mu - 2}}{(1 - x^2)^{{\frac12\mu}}} \Bigg[2x
 \big(\expe^{\frac12 {\rm i}\pi(\nu + \mu)} +\expe^{-\frac12 {\rm i}\pi(\nu + \mu)} \big)
 \frac{\Ga\big(\frac{\nu+\mu+2}{2}\big)} {\Ga\big(\frac{\nu-\mu+1}{2}\big)}
 \F{\frac{-\nu-\mu+1}{2}, \frac{\nu-\mu+2}{2}}{\frac32}{x^2}
 \\ \hphantom{\Q_\nu^\mu (x) = \frac{\pi^\frac12 2^{\mu - 2}}{(1 - x^2)^{{\frac12\mu}}}\Bigg[}
{} +
\big(\expe^{-\frac12 {\rm i}\pi(\nu + \mu + 1)}+\expe^{\frac12 {\rm i}\pi(\nu + \mu + 1)}
\big)\frac{\Ga\big(\frac{\nu+\mu+1}{2}\big)}{\Ga\big(\frac{\nu-\mu+2}{2}\big)}
\F{\frac{-\nu-\mu}{2}, \frac{\nu-\mu+1}{2}}{\frac12}{x^2}\!\Bigg].
\end{gather*}
Since this is equivalent to the claimed result, we have completed the proof.
\end{proof}

\begin{thm}\label{tII4}
Let $\nu,\mu\in\C$ such that $\nu+\mu\notin-\N$ and $\nu+\frac12\notin\Z$. If $\pm \mathop{\rm Im} x>0$ then
\begin{gather*}
\Q_\nu^\mu(x)=\sqrt\pi\, 2^{-\nu-1} \expe^{\pm {\rm i}\pi\mu}\left(\cos(\pi\mu)
\mp {\rm i}\frac{\sin(\pi(\mu-\nu))}{2\cos(\pi\nu)}
\right)\frac{\Ga(\nu+\mu+1)}{\Ga\big(\nu+\frac32\big)}
\\ \hphantom{\Q_\nu^\mu(x)=\sqrt\pi}
{}\times x^{{-\nu-\mu-1}}\big(1-x^2\big)^{\frac12\mu} \F{\frac{\nu+\mu+1}2,\frac{\nu+\mu+2}2}{\nu+\frac32}{\frac{1}{x^2}}
\\ \hphantom{\Q_\nu^\mu(x)=}
{}+ \pi^{\frac32} 2^{\nu-1} \frac{\expe^{\pm\pi {\rm i}\left(\frac12+\mu\right)}\sec(\pi\nu)}{\Ga(\nu-\mu+1)\Ga\big(\frac12-\nu\big)}
 x^{\nu-\mu}\big(1-x^2\big)^{\frac12\mu}
\F{\frac{\mu-\nu}2,\frac{\mu-\nu+1}{2}}{\frac12-\nu}{\frac{1}{x^2}}\!.
\end{gather*}
\end{thm}
\begin{proof}
We represent $\Q_\nu^\mu(x)$ by~\eqref{FerrersQc}
if $\mathop{\rm Im} x>0$ and by~\eqref{FerrersQf} if~$\mathop{\rm Im} x<0$.
Then we use~\cite[Entry 28, p.~162]{MOS} which agrees with \eqref{defQ},
and the desired result follows.
\end{proof}

\begin{thm}\label{tII5}
Let $x\in D_1^+$, $\nu \in \C$, $\mu \in \C \setminus \Z$ such
that $\nu + \mu \notin -\N$. Then
\begin{gather*}
\Q_\nu^\mu (x) =
\frac{2^{\mu-1}\cos(\pi\mu)\Ga(\mu)}{\big(1-x^2\big)^{{\frac12\mu}}} x^{\nu+\mu}
\F{\frac{-\nu-\mu}{2}, \frac{-\nu-\mu+1}{2}}{1-\mu}{\frac{x^2-1}{x^2}}
\\ \hphantom{\Q_\nu^\mu (x) =}
{} +\frac{\Ga(\nu+\mu+1) \Ga(-\mu)}{2^{\mu+1}\Ga(\nu-\mu+1)}\big(1-x^2\big)^{{\frac12\mu}} x^{\nu-\mu}
\F{\frac{\mu-\nu}{2}, \frac{\mu-\nu+1}{2}}{1+\mu}{\frac{x^2-1}{x^2}}\!.
\end{gather*}
\end{thm}

\begin{proof}
The hypergeometric representation~\cite[Entry 29, p.~162]{MOS},{\samepage
\begin{gather*}
\expe^{-{\rm i}\pi\mu} Q_\nu^\mu (z) = \frac{2^{\mu-1}\Ga(\mu)z^{\nu+\mu}}
{(z^2-1)^{\frac12\mu}}
\F{\frac{-\nu-\mu}{2},\frac{1-\nu-\mu}{2}}{1-\mu}{1-\frac{1}{z^2}}
\\ \hphantom{\expe^{-{\rm i}\pi\mu} Q_\nu^\mu (z) =}
{} +\frac{\Ga(\nu+\mu+1)\Ga(-\mu)z^{\nu-\mu}(z^2-1)^{\frac12\mu}}
{2^{\mu+1}\Ga(\nu-\mu+1)}
\F{\frac{\mu-\nu}{2},\frac{\mu-\nu+1}{2}}{1+\mu}{1-\frac{1}{z^2}}\!,
\end{gather*}
is valid for $\mathop{\rm Re} z\!>\!0$, $z\not\in(0,1]$ $\big($then $1-\frac1{z^2}\notin [1,\infty)\big)$.
Using~\eqref{FerrersQ}, the result follows.}
\end{proof}

\begin{thm}\label{tII6}
Let $x\in D_1$, $\mu,\nu\in\C$ such that $\nu+\mu\notin \N$. Then
\begin{gather*}
{\sf Q}_\nu^\mu(x)=\sqrt{\pi}\,2^\mu\Bigg[\frac{-\Ga\big(\frac{\nu+\mu+1}{2}\big) \sin\big(\frac{\pi}{2}(\nu+\mu)\big)}{2\Ga\big(\frac{\nu-\mu+2}{2}\big)\big(1-x^2\big)^{\tfrac12(\nu+1)}}
\F{\frac{\nu-\mu+1}{2},\frac{\nu+\mu+1}{2}}{\frac12}{\frac{x^2}{x^2-1}}
\\ \hphantom{{\sf Q}_\nu^\mu(x)=}
{}+\frac{\Ga\big(\frac{\nu+\mu+2}{2}\big)
\cos\big(\frac{\pi}{2}(\nu+\mu)\big)x\big(1-x^2\big)^{{\frac12(\nu-1)}}}
{\Ga\big(\frac{\nu-\mu+1}{2}\big)}
\F{\frac{\mu-\nu+1}{2},\frac{-\nu-\mu+1}{2}}{\frac32}{\frac{x^2}{x^2-1}}\!\Bigg].
\end{gather*}
\end{thm}
\begin{proof}
The hypergeometric representation of the associated Legendre function of the second kind~\cite[Entry 30, p.~162]{MOS},
\begin{gather*}
\expe^{-{\rm i}\pi\mu} Q_\nu^\mu (z) = \frac{\sqrt{\pi}\, 2^\mu \Ga\big(\frac{\nu+\mu+2}{2}\big)}
{\Ga\big(\frac{\nu-\mu+1}{2}\big)}
\expe^{\mp \frac12 {\rm i} \pi\left(\nu-\frac12\right)}z(z^2 - 1)^{{\frac12(\nu-1)}}
\F{\frac{-\nu-\mu+1}{2}, \frac{\mu-\nu+1}{2}}{\frac32}{\frac{z^2}{z^2-1}}
\\ \hphantom{\expe^{-{\rm i}\pi\mu} Q_\nu^\mu (z) =}
{} + \frac{\sqrt{\pi}\, 2^{\mu - 1} \Ga\big(\frac{\nu+\mu+1}{2}\big)}{\Ga\big(\frac{\nu-\mu+1}{2}\big)}
\expe^{\mp \frac12 {\rm i} \pi\left(\nu+\frac12\right)} (z^2 - 1)^{{\frac12\nu}}
\F{\frac{-\nu-\mu}{2}, \frac{\mu-\nu}{2}}{\frac12}{\frac{z^2}{z^2 - 1}}\!,
\end{gather*}
holds with the upper and lower sign chosen according to $\pm \mathop{\rm Im} z>0$.
Then \eqref{FerrersQ} yields the desired result.
\end{proof}

\begin{rem}
Theorem~\ref{tII3} (or~\ref{tII6}) writes $\Q_\nu^\mu(x)$
as a sum of an even and an odd solution of~the associated Legendre equation.
\end{rem}

\begin{rem}
The hypergeometric representations of $\Q_\nu^\mu(x)$ given in the previous six theorems can be written
in slightly different forms by applying Euler's transformation~\eqref{Euler}. For instance, the result of Theorem~\ref{tII3} can be written as
\begin{gather*}
\Q_\nu^\mu(x)=\sqrt{\pi}\,2^{\mu-1}\big(1-x^2\big)^{{\frac12\mu}}\Biggl[
\frac{-\sin\big(\tfrac{\pi}{2}(\nu+\mu)\big)
\Ga\big(\frac{\nu+\mu+1}{2}\big)}
{\Ga\big(\frac{\nu-\mu+2}{2}\big)}
\F{\frac{\nu+\mu+1}{2},\frac{\mu-\nu}{2}}{\frac12}{x^2}
\\[1ex] \hphantom{\Q_\nu^\mu(x)=\sqrt{\pi}\,2^{\mu-1}\big(1-x^2\big)^{{\frac12\mu}}\Biggl[}
{}+\frac{2\cos\big(\tfrac{\pi}{2}(\nu+\mu)\big)
\Ga\big(\frac{\nu+\mu+2}{2}\big)}
{\Ga\big(\frac{\nu-\mu+1}{2}\big)}
x\F{\frac{\nu+\mu+2}{2},\frac{\mu-\nu+1}{2}}{\frac32}{x^2\!}\Biggr].
\end{gather*}
\end{rem}

\begin{rem}
We observe that by applying Pfaff's transformation \eqref{Pfaff1} and~\eqref{Pfaff2}, the results of Theorems~\ref{tII5},~\ref{tII4} and~\ref{tII6}
follow immediately from those of Theorems~\ref{tII1},~\ref{tII2} and~\ref{tII3}, respectively.
The hypergeometric representation of $\Q_\nu^\mu(x)$ stated in Theorem~\ref{tII1} (or~\ref{tII5}) is of~the form~\eqref{connect1}
with the Ferrers function $\P_\nu^\mu(x)$ replaced by {the} hypergeometric representation
\begin{gather*}
\P_\nu^\mu(x)=\frac{2^\mu}{\Ga(1-\mu)}\big(1-x^2\big)^{-\frac12\mu}
\F{-\frac12\mu-\frac12\nu,\frac12-\frac12\mu+\frac12\nu}{1-\mu}{1-x^2}\!.
\end{gather*}
\end{rem}

\begin{rem}
Theorem~\ref{tII2} expresses $\Q_\nu^\mu(x)$ in terms of the associated Legendre functions $Q_\nu^\mu(x)$ and $Q_{-\nu-1}^\mu(x)$.
The argument of the hypergeometric function is $w_8=1/\big(1-x^2\big)$.
If~one considers $x\in(-1,1)$, then this function maps to $(1,\infty)$, where the Gauss hypergeometric
function takes two values depending on whether the value is approached from above or
below the ray $[1,\infty)$. Hence, to compute these values, one must use Theorems
\ref{Theorem1}--\ref{Theorem4}.
Then the arguments of the hypergeometric functions are transformed to
$1-x^2$, $x^2/\big(x^2-1\big)$, $x^2$, \mbox{$\big(x^2-1\big)/x^2$}, and
we obtain {the} results that are already contained in Theorems~\ref{tII1},~\ref{tII6},~\ref{tII3} and~\ref{tII5}, respectively.
Similar remarks apply to Theorem~\ref{tII4}.
\end{rem}

\section[Group III hypergeometric representations for Q nu mu (x)]
{Group III hypergeometric representations for $\boldsymbol{\Q_\nu^\mu(x)}$}\label{sec5}

In this section we {focus on} hypergeometric functions listed in
Section~\ref{sec1}, with arguments $w_j$, $j=13,\dots,18$.
Since we want to represent {the} Ferrers functions (defined on $D_1$), we replace $\sqrt{x^2-1}$ by
${\rm i}\sqrt{1-x^2}$. The latter is an analytic function on $D_1$. We note that ${\rm i}\sqrt{1-x^2}=\pm \sqrt{x^2-1}$
if~$\pm\mathop{\rm Im} x>0$.

\medskip

The following result in terms of trigonometric functions appears
in~\cite[p.~168]{MOS} in a different form.

\begin{thm}\label{tIII1}
Let $x\in D_1$, $\nu,\mu\in\C$, $\nu+\mu\notin-\N$,
$\nu+\frac12\notin\Z$. Then
for the upper sign chosen
everywhere as well as with the lower
sign chosen everywhere,
\begin{gather*}
\Q_\nu^\mu(x)=\frac{\sqrt\pi}{2^\frac32\big(1-x^2\big)^\frac14}
\\ \hphantom{\Q_\nu^\mu(x)=}
{}\times\Biggl[
\expe^{\pm\frac{{\rm i}\pi}{2}\left(\mu+\frac12\right)}
\frac{\Ga\big(\nu+\frac12\big)}{\Ga(\nu-\mu+1)}
\big(x\pm {\rm i}\sqrt{1\!-x^2}\big)^{\nu+\frac12}
\F{\frac12\!+\mu,\frac12\!-\mu}{\frac12-\nu}
{\frac{\mp x+{\rm i}\sqrt{1\!-x^2}}{2{\rm i}\sqrt{1-x^2}}}
\\ \hphantom{\Q_\nu^\mu(x)=\times\Biggl[}
{}+ \expe^{\mp\frac{{\rm i}\pi}{2}\left(\mu+\frac12\right)}
\frac{\Ga(\nu+\mu+1)}{\Ga\big(\nu+\frac32\big)}
\left(1+\expe^{\pm {\rm i}\pi(\nu+\mu)}\frac{\cos(\pi\mu)}{\cos(\pi\nu)}\right)
\big(x\mp {\rm i}\sqrt{1-x^2}\big)^{\nu+\frac12}
\\ \hphantom{\Q_\nu^\mu(x)=\times\Biggl[+}
{}\times\F{\frac12+\mu,\frac12-\mu}{\nu+\frac32}
{\frac{\mp x+{\rm i}\sqrt{1-x^2}}{2{\rm i}\sqrt{1-x^2}}}\!\Bigg],
\end{gather*}
or equivalently with $\mathop{\rm Re}\theta\in(0,\pi)$,
\begin{gather*}
\Q_\nu^\mu(\cos\theta)=\frac{\sqrt{\pi}}
{2^\frac32\sqrt{\sin\theta}}\Biggl[
\expe^{\pm\frac{{\rm i}\pi}{2}\left(\mu+\frac12\right)}
\frac{\Ga\big(\nu+\frac12\big)}{\Ga(\nu-\mu+1)}
\expe^{\pm {\rm i}\left(\nu+\frac12\right)\theta}
\F{\frac12+\mu,\frac12-\mu}{\frac12-\nu}{\frac12\pm\frac{\rm i}{2}\cot\theta}
\\ \hphantom{\Q_\nu^\mu(\cos\theta)=}
{} + \expe^{\mp\frac{{\rm i}\pi}{2}\left(\mu+\frac12\right)}
\frac{\Ga(\nu+\mu+1)}{\Ga\big(\nu+\frac32\big)}
\left(1+\expe^{\mp{\rm i}\pi(\nu+\mu)}\frac{\cos(\pi\mu)}{\cos(\pi\nu)}\right)
\\ \hphantom{\Q_\nu^\mu(\cos\theta)=+}
{}\times
\expe^{\mp {\rm i}\left(\nu+\frac12\right)\theta}
\F{\frac12+\mu,\frac12-\mu}{\nu+\frac32}
{\frac12\pm\frac{\rm i}{2}\cot\theta}\Biggr].
\end{gather*}
\end{thm}
\begin{proof}
We first prove the result with the {upper} sign.
We claim that
\begin{gather*}
w_{13}=\frac{-x+{\rm i}\sqrt{1-x^2}}{2{\rm i}\sqrt{1-x^2}}
\end{gather*}
is a conformal map from $D_1$ to the complex plane cut along the rays $(-\infty,0]$ and $[1,\infty)$.
To see this let $x=\cos\theta$. This is a conformal map from the strip $S=\{\theta\colon \mathop{\rm Re}\theta<\pi\}$ onto $D_1$.
Then
\begin{gather*}
w_{13}=\frac12 +\frac{\rm i}{2}\cot\theta .
\end{gather*}
Now $\cot\theta$ is a conformal map from $S$ to $\C\setminus((- {\rm i}\infty,-{\rm i}]\cup [{\rm i}, {\rm i}\infty))$ which establishes the claim.
Hence, using the principal value of the hypergeometric function, the composition of the hypergeometric function
with $w_{13}$
is an analytic function on $D_1$. Further, we note
that $x\pm {\rm i}\sqrt{1-x^2}\notin (-\infty,0]$ for $x\in D_1$.
Therefore, the right-hand side of the stated identity is an analytic function on $D_1$, so
it is sufficient to prove this identity for $\mathop{\rm Im} x>0$.
For $\mathop{\rm Im} z>0$, we insert
the hypergeometric representation of the associated
Legendre function~\cite[Entry 31, p.~162]{MOS}
\begin{gather}\label{MOS31}
\expe^{-{\rm i}\pi\mu} Q_\nu^\mu (z) = \frac{\sqrt{\pi}\,\Ga(\nu\!+\mu\!+1)
\big(z\!-\sqrt{z^2\!-1}\big)^{\nu+\frac12}}{\sqrt{2}\Ga\big(\nu+\frac32\big)(z^2\!-1)^\frac14}
\F{\frac12+\mu,\frac12\!-\mu}{\nu+\frac32}{\frac{-z+\sqrt{z^2-1}}{2\sqrt{z^2-1}}}\!,
\end{gather}
in~\eqref{FerrersQc},
We obtain the desired result by using \eqref{Eulerreflect} and the identity
\begin{gather*}
2\expe^{{\rm i}\pi\mu}\left(\sin(\pi\nu)- {\rm i}\frac{\sin(\pi(\mu-\nu))}{2\cos(\pi\nu)}\right)=1+\expe^{{\rm i}\pi(\nu+\mu)}\frac{\cos(\pi\mu)}{\cos(\pi\nu)} .
\end{gather*}
The proof of the result with the lower sign is similar.
The function
\begin{gather*}
w_{17}=\frac{x+{\rm i}\sqrt{1-x^2}}{2{\rm i}\sqrt{1-x^2}}{,}
\end{gather*}
is also a conformal map from $D_1$ to the complex plane cut along the rays $(-\infty,0]$ and $[1,\infty)$.
Hence the right-hand side of the claimed identity is again an analytic function on $D_1$, so it is sufficient to prove it for $\mathop{\rm Im} x>0$ or $\mathop{\rm Im} x<0$.
The desired result follows from \eqref{FerrersQc} and the hypergeometric representation of the associated
Legendre function cf.~\cite[Entry 35, p.~162]{MOS} in the half-plane $\mathop{\rm Im} x>0$. However, it is
easier to use \eqref{FerrersQf}
and \eqref{MOS31} in the half-plane $\mathop{\rm Im} x<0$. In~this half-plane $\sqrt{x^2-1}=-{\rm i}\sqrt{1-x^2}$ and the desired representation follows. In~order to obtain the trigonometric form of the representation set $x=\cos\theta$ with
$\mathop{\rm Re}\theta\in(0,\pi)$
and note that ${\rm i}\sqrt{1-x^2}=\sin\theta$.
This completes the proof.
\end{proof}

The following result in terms of trigonometric functions appears
in~\cite[p.~168]{MOS} in a different form.

\begin{thm}\label{tIII2}
Let $x\in D_1$, $\nu,\mu\in\C$, $\nu+\mu\notin-\N$,
$\nu+\frac12\notin\Z$. Then for the upper sign chosen
everywhere as well as with the lower sign chosen everywhere,
\begin{gather*}
\Q_\nu^\mu(x)=\sqrt{\pi}\,2^{\mu-1} \big(1-x^2\big)^{{\frac12\mu}}
\\ \hphantom{\Q_\nu^\mu(x)=}
{}\times\Biggl[\expe^{\pm {\rm i}\pi\left(\mu+\frac12\right)}
\frac{\Ga\big(\nu+\frac12\big)}{\Ga(\nu-\mu+1)}
\big(x\pm {\rm i}\sqrt{1-x^2}\big)^{\nu-\mu}
\F{\frac12+\mu,\mu-\nu}{\frac12-\nu}
{\frac{x\mp {\rm i}\sqrt{1-x^2}}{x\pm {\rm i}\sqrt{1-x^2}}}
\\ \hphantom{\Q_\nu^\mu(x)=\times\Biggl[}
{}+ \frac{\Ga(\nu+\mu+1)}{\Ga\big(\nu+\frac32\big)}
\left(1+\expe^{\pm {\rm i}\pi(\nu+\mu)}\frac{\cos(\pi\mu)}{\cos(\pi\nu)}\right)
\big(x\mp {\rm i}\sqrt{1-x^2}\big)^{\nu+\mu+1}
\\ \hphantom{\Q_\nu^\mu(x)=\times\Biggl[+}
{}\times
\F{\frac12+\mu,\nu+\mu+1}{\nu+\frac32}
{\frac{x\mp {\rm i}\sqrt{1-x^2}}{x\pm {\rm i}\sqrt{1-x^2}}}\!\Biggr],
\end{gather*}
or equivalently with $\mathop{\rm Re}\theta\in(0,\pi)$,
\begin{gather*}
\Q_\nu^\mu(\cos\theta)=
\sqrt{\pi}\,2^{\mu-1}(\sin\theta)^\mu
\Biggl[\expe^{\pm {\rm i}\pi\left(\mu+\frac12\right)}
\frac{\Ga\big(\nu+\frac12\big)}{\Ga(\nu-\mu+1)}
\expe^{\pm {\rm i}(\nu-\mu)\theta}
\F{\frac12+\mu,\mu-\nu}{\frac12-\nu}
{\expe^{\mp 2{\rm i}\theta}}
\\ \hphantom{\Q_\nu^\mu(\cos\theta)=}
{}+ \!\frac{\Ga(\nu\!+\!\mu\!+\!1)}{\Ga\big(\nu+\frac32\big)}
\left(\!1\!+\!\expe^{\pm {\rm i}\pi(\nu+\mu)}\frac{\cos(\pi\mu)}{\cos(\pi\nu)}\right)
\expe^{\mp {\rm i}(\nu+\mu+1)\theta}
\F{\frac12\!+\!\mu,\nu\!+\!\mu\!+\!1}{\nu+\frac32}
{\expe^{\mp 2{\rm i}\theta}}\!\Biggr].
\end{gather*}
\end{thm}

\begin{proof}
We first prove the result with the upper sign.
We note that the function
\begin{gather*}
w_{14}=\frac{x-{\rm i}\sqrt{1-x^2}}{x+{\rm i}\sqrt{1-x^2}}=\frac{w_{13}}{w_{13}-1}
\end{gather*}
is a conformal map from $D_1$ to the complex plane cut along the ray $[0,\infty)$.
Hence the right-hand side of the stated identity is an analytic function on $D_1$, so it is sufficient to prove it for $\mathop{\rm Im} x>0$.
Now we insert the hypergeometric representation~\cite[Entry 32, p.~162]{MOS}
\begin{gather}
\expe^{-{\rm i}\pi\mu} Q_\nu^\mu (z)=\sqrt{\pi}\,2^\mu
\frac{\Ga(\nu+\mu+1)}{\Ga\big(\mu+\frac32\big)}(z^2-1)^{\frac12\mu}
\big(z+\sqrt{z^2-1}\big)^{\nu+\mu+1}\nonumber
\\ \hphantom{\expe^{-{\rm i}\pi\mu} Q_\nu^\mu (z)=}
{}\times\F{\frac12+\mu,\nu+\mu+1}{\nu+\frac32}{\frac{z-\sqrt{z^2-1}}{z+\sqrt{z^2-1}}}\label{MOS32}
\end{gather}
in \eqref{FerrersQc}, and obtain the desired result.
To prove the result with the lower sign we can either use the hypergeometric representation~\cite[Entry 36, p.~163]{MOS}
of $Q_\nu^\mu(z)$ in the half-plane $\mathop{\rm Im} z>0$ or \eqref{MOS32} in the half-plane $\mathop{\rm Im} z<0$.
This completes the proof.
\end{proof}

\begin{thm}\label{tIII3}
Let $x\in D_1^+$, $\nu,\mu\in\C$, $2\mu\not\in\Z$, $\nu+\mu\notin-\N$. Then
for the upper sign chosen everywhere as well as with the lower sign chosen everywhere,
\begin{gather*}
\Q_\nu^\mu(x)=\frac{\Ga(-\mu)}{2^{\mu+1}}
\frac{\Ga(\nu\!+\!\mu\!+\!1)}{\Ga(\nu\!-\!\mu\!+\!1)}\big(1-x^2\big)^{\frac12\mu}
\big(x\pm {\rm i}\sqrt{1\!-\!x^2}\big)^{\nu-\mu}
\F{\frac12\!+\!\mu,\mu\!-\!\nu}{1+2\mu}
{\frac{2{\rm i}\sqrt{1-x^2}}{\pm \, x\!+\!{\rm i}\sqrt{1\!-\!x^2}}}\!
\\ \hphantom{\Q_\nu^\mu(x)=}
{}+2^{\mu-1}\Gamma(\mu)
\frac{\cos(\pi\mu)\big(x\pm {\rm i}\sqrt{1-x^2}\big)^{\nu+\mu}}
{\big(1-x^2\big)^{{\frac12\mu}}}\F{\frac12-\mu,-\nu-\mu}{1-2\mu}
{\frac{2{\rm i}\sqrt{1-x^2}}{\pm\,x+{\rm i}\sqrt{1-x^2}}}\!,
\end{gather*}
or equivalently with $\mathop{\rm Re}\theta\in(0,\pi)$,
\begin{gather*}
\Q_\nu^\mu(\cos\theta)=
\frac{\Ga(-\mu)}{2^{\mu+1}}
\frac{\Ga(\nu+\mu+1)}{\Ga(\nu-\mu+1)}(\sin\theta)^{\mu}
\expe^{\pm {\rm i}(\nu-\mu)\theta}\F{\frac12+\mu,\mu-\nu}{1+2\mu}
{1-\expe^{\mp 2{\rm i}\theta}}\nonumber
\\ \hphantom{\Q_\nu^\mu(\cos\theta)=}
{}+2^{\mu-1}\Gamma(\mu)
\frac{\cos(\pi\mu)\expe^{\pm {\rm i}(\nu+\mu)\theta}}
{(\sin\theta)^{\mu}}\F{\frac12-\mu,-\nu-\mu}{1-2\mu}
{1-\expe^{\mp 2{\rm i}\theta}}\!.
\end{gather*}
\end{thm}
\begin{proof}
We first prove the result with the upper sign.
We observe that the function
\begin{gather*}
w_{15}=\frac{2{\rm i}\sqrt{1-x^2}}{x+{\rm i}\sqrt{1-x^2}}=\frac{1}{w_{17}}{,}
\end{gather*}
maps the imaginary axis to the branch cut $[1,\infty)$ of the hypergeometric function.
Then $w_{15}$ is a conformal map from~$D_1^+$ to the upper half-plane.
Therefore, the right-hand side of the stated identity is analytic on~$D_1^+$
so it is sufficient to prove it for $x\in(0,1)$.
We obtain the desired result by using~\eqref{FerrersQ}
and the hypergeometric representation of $Q_\nu^\mu(x)$ given in~\cite[Entry 33, p.~163]{MOS}.
However, it is much simpler to use~\eqref{FerrersQb} and the hypergeometric representation~\cite[Entry 15, p.~158]{MOS}
\begin{gather*}
P_\nu^\mu(z)={\frac{2^\mu\big(z+\sqrt{z^2-1}\big)^{\nu+\mu}}{\Ga(1-\mu)(z^2-1)^{\frac12\mu}}}
\F{-\nu-\mu,\frac12-\mu}{1-2\mu}{\frac{2\sqrt{z^2-1}}{z+\sqrt{z^2-1}}}\!.
\end{gather*}

In order to derive the result with the lower sign, use~\cite[Entry 34, p.~163]{MOS} or~\eqref{FerrersQe} and~\eqref{MOS32} in $\mathop{\rm Re} x>0$, $\mathop{\rm Im} x<0$, respectively.
This completes the proof.
\end{proof}

\begin{rem}
The hypergeometric representations of $\Q_\nu^\mu(x)$ given in the previous three theorems can be written
in slightly different forms by using Euler's transformation~\eqref{Euler}.
We observe that by applying Pfaff's transformations~\eqref{Pfaff1} and~\eqref{Pfaff2}, Theorem~\ref{tIII2} follows from Theorem~\ref{tIII1}.
Moreover, the result of Theorem~\ref{tIII3} with the lower sign follows from the same theorem with the upper sign.
\end{rem}

\begin{rem}
Theorem~\ref{tIII1} with the upper sign represents $\Q_\nu^\mu(x)$ in the upper half-plane as~a~li\-near combination of $Q_\nu^\mu(x)$ and $Q_{-\nu-1}^\mu(x)$ (cf.\ \eqref{FerrersQc}). In~particular, if $x\in(-1,1)$, $\Q_\nu^\mu(x)$ is given by a linear combination of $Q_\nu^\mu(x+{\rm i}0)$
and $Q_{-\nu-1}^\mu(x+{\rm i}0)$. The argument $w_{13}$, \mbox{$x\in(-1,1)$}, of
the hypergeometric functions lies on the line $\mathop{\rm Re} w=\frac12$.
Theorem~\ref{tIII1} with the lower sign represents~$\Q_\nu^\mu(x)$ in the lower half-plane as
a linear combination of $Q_\nu^\mu(x)$ and $Q_{-\nu-1}^\mu(x)$. In~particular, if~$x\in(-1,1)$, $\Q_\nu^\mu(x)$ is given by a linear combination of $Q_\nu^\mu(x-{\rm i}0)$
and $Q_{-\nu-1}^\mu(x-{\rm i}0)$. Again the argument $w_{17}$, $x\in(-1,1)$, of
the hypergeometric functions lies on the line $\mathop{\rm Re} w=\frac12$.
Similar results hold regarding Theorem~\ref{tIII2}. In~this case the argument
$w_{14}$, $x\in(-1,1)$, of~the hypergeometric functions lies on the unit circle $|w|=1$ excluding $w=1$.
Theorem~\ref{tIII3} represents $\Q_\nu^\mu(x)$ as a linear combination \eqref{connect1} of $\P_\nu^{-\mu}(x)$ and $\P_\nu^\mu(x)$ for
$x\in D_1^+$,
using an appropriate hypergeometric representation
of the Ferrers functions of the first kind. The argument $w_{15}$, $x\in(0,1)$, lies on the circle $|w-1|=1$ excluding $w=2$.
The range of values of $w_j$, $j=13,\dots,18$, $x\in(-1,1)$, is depicted in Figure~\ref{fig1}.
\end{rem}

\begin{figure}[htb!]\centering
\includegraphics[scale=.95]{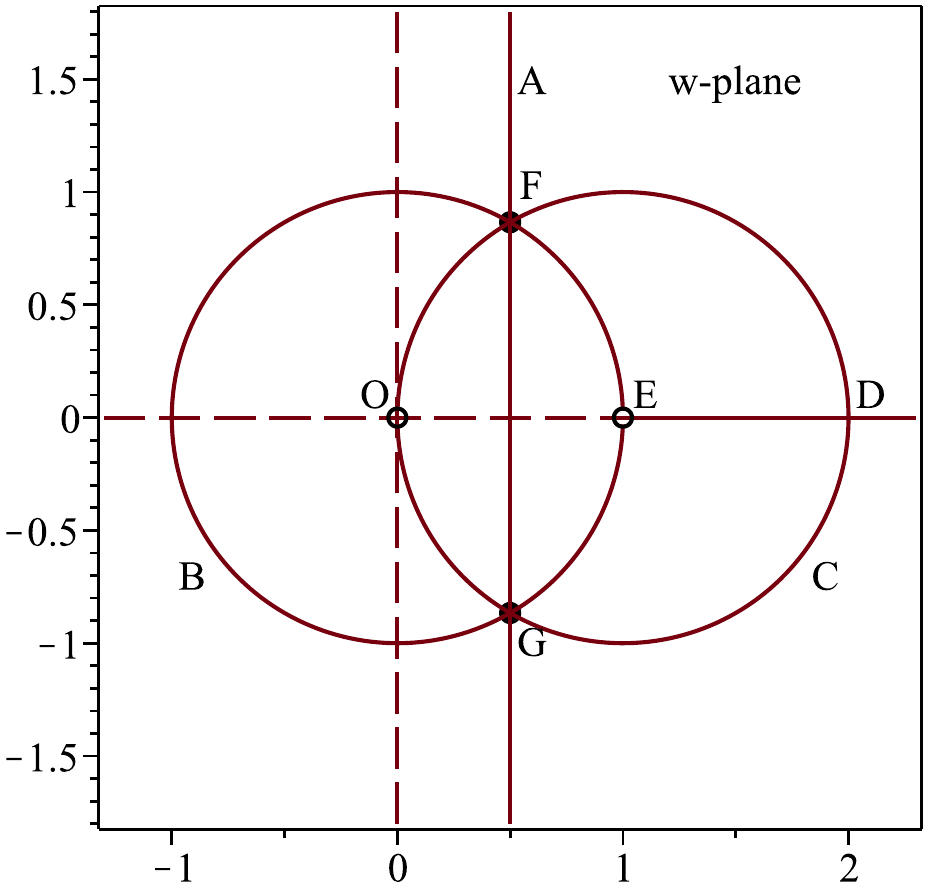}
\caption{This figure represents a plot in~the complex plane of the complex arguments
which appear in~Theorems~\ref{tIII1},~\ref{tIII2} and~\ref{tIII3} when $x\in(-1,1)$.
The symbol $O$ represents the origin of the complex $w$-plane, with the dashed horizontal
and vertical lines emanating from it being the real and imaginary axes respectively.
The symbol $A$ represents the line parallel to the imaginary
axis $w=\frac12\pm \frac{\rm i}{2}\cot\theta$ in~Theorem~\ref{tIII1}, where
$\theta\in(0,\pi)$.
The symbol $B$ represents the unit circle centered at~the origin~excluding
the point located at~$w=1$ in~Theorem~\ref{tIII2}.
The symbol $C$ represents the unit circle centered at~the point $w=1$ excluding
the point located at~$w=0$ in~Theorem~\ref{tIII3}.
The point $w=2$ ($D$) corresponding to $x=0$ is excluded in~Theorem~\ref{tIII3}.
The symbol $E$ represents the branch cut of the hypergeometric function
located at~$w\in[1,\infty)$.
The symbols $F$ and $G$ represent the two points where all three curves intersect
each other at~$w=\frac12\pm \frac{\sqrt{3}}{2}{\rm i}$ respectively.}\label{fig1}
\end{figure}

\section{{A} Fourier series representation}\label{sec6}

A Fourier series for a function
$f(x)$ is given by the infinite
sum
\begin{gather*}
f(x)=\sum_{n=-\infty}^\infty a_n \expe^{{\rm i}n\theta},
\end{gather*}
where the coefficients $a_n$ must
satisfy certain specific conditions to
guarantee convergence, see below. In~this section we perform
analysis for a Fourier series
representation of the Ferrers function
of the second kind.

\begin{thm}\label{Fourier1}
Let $x\in D_1$ and $\nu,\mu\in\C$ such that $\nu+\mu\not\in -\N$. Then
\begin{gather*}
\Q_\nu^\mu(x)= \sqrt\pi\, 2^{\mu-1}\big(1-x^2\big)^{\frac12\mu} \frac{\Ga(\nu+\mu+1)}{\Ga\big(\nu+\frac32\big)}\nonumber
\\ \hphantom{\Q_\nu^\mu(x)=}
{}\times\Bigg[u^{\nu+\mu+1}\F{\mu+\frac12,\nu+\mu+1}{\nu+\frac32}{\frac{u}{v}}
+v^{\nu+\mu+1}\F{\mu+\frac12,\nu+\mu+1}{\nu+\frac32}{\frac{v}{u}}\!\Bigg],
\end{gather*}
where
\begin{gather*}
u := x+{\rm i}\sqrt{1-x^2},\qquad v := x-{\rm i}\sqrt{1-x^2}=\frac1u.
\end{gather*}
\end{thm}
\begin{proof}
Both sides of the stated identity are analytic functions on $D_1$, so it is sufficient to prove it for $x\in(-1,1)$.
If $x\in(-1,1)$ we use \eqref{FerrersQ} with $Q_\nu^\mu(x+{\rm i}0)$ and $Q_\nu^\mu (x-{\rm i}0)$ both
expressed through \eqref{MOS32}. Then the desired identity follows noting that $\sqrt{(x\pm {\rm i}0)^2-1}=\pm {\rm i}\sqrt{1-x^2}$.
\end{proof}

If we set $x=\cos\theta$, {$\mathop{\rm Re}\theta\in(0,\pi)$},
then Theorem~\ref{Fourier1} gives
\begin{gather*}
\Q_\nu^\mu(\cos\theta)= \sqrt\pi\, 2^{\mu-1}(\sin\theta)^\mu\frac{\Ga(\nu+\mu+1)}{\Ga\big(\nu+\frac32\big)}
\\ \qquad
{}\times\Bigg[\expe^{{\rm i}(\nu\!+\mu\!+1)\theta}\F{\mu\!+\frac12,\nu\!+\mu\!+1}{\nu+\frac32}{\expe^{2{\rm i}\theta}}
\!+\expe^{-{\rm i}(\nu+\mu+1)\theta}\F{\mu+\frac12,\nu+\mu+1}{\nu+\frac32}{\expe^{-2{\rm i}\theta}}\!\Bigg].
\end{gather*}
If $\theta\in(0,\pi)$ then the arguments $w=\expe^{\pm 2{\rm i}\theta}$ of the hypergeometric function lie on the unit circle $|w|=1$.
Provided the hypergeometric series converges at $\expe^{\pm 2{\rm i}\theta}$ we obtain (using Abel's theorem on power series)~\cite[equation~(14.13.2)]{NIST:DLMF}
\begin{gather}\label{Fourier}
\Q_\nu^\mu(\cos\theta)= \sqrt\pi\,2^\mu(\sin\theta)^\mu\sum_{k=0}^\infty \frac{\Ga(\nu+\mu+k+1)}{\Ga\big(\nu+k+\frac32\big)}
\frac{\big(\mu+\frac12\big)_k}{k!} \cos((\nu+\mu+2k+1)\theta) .
\end{gather}
Regarding the convergence of the series in~\eqref{Fourier} we have the following result. Statement $(b)$ of~Theo\-rem~\ref{Fourierconvergence} is more precise than the corresponding statement in~\cite[Section~14.13]{NIST:DLMF}.

\begin{thm}\label{Fourierconvergence}
Let $\theta\in(0,\pi)$, $\nu,\mu\in\C$ such that $\nu+\mu\in\C\setminus-\N$.
\begin{enumerate}\itemsep=0pt

\item[$(a)$] If $\mathop{\rm Re}\mu<0$ then the series in \eqref{Fourier} converges absolutely.

\item[$(b)$] If $0\le \mathop{\rm Re} \mu<\frac12$ then the series in \eqref{Fourier} converges, but, if $\theta\ne\frac12\pi$,
it does not converge absolutely.

\item[$(c)$] If $\mathop{\rm Re}\mu\ge \frac12$ and $\theta\ne \frac12\pi${,} then the series in \eqref{Fourier} diverges.
\end{enumerate}
\end{thm}
\begin{proof}
$(a)$ It is known~\cite[Section~15.2(i)]{NIST:DLMF} that the Gauss hypergeometric series $\F{a,b}{c}{w}$ converges absolutely on the unit circle $|w|=1$ if $\mathop{\rm Re}(c-a-b)>0$. In~our case $a=\mu+\frac12$, $b=\nu+\mu+1$, $c=\nu+\frac32$ so
$c-a-b=-2\mu$. If $\mathop{\rm Re}\mu<0$ it follows that the series in \eqref{Fourier} is the sum of two absolutely convergent series and so is
itself absolutely convergent.

\looseness=-1
$(b)$ Suppose that $0\le\mathop{\rm Re}\mu<\frac12$.
If $-1<\mathop{\rm Re}(c-a-b)\le 0$ then the Gauss hypergeometric series converges conditionally at $|w|=1$, $w\ne 1$~\cite[Section~15.2(i)]{NIST:DLMF}.
It follows that the series in \eqref{Fourier} is the sum of two convergent series and so is itself convergent.
However, it is not true that the sum of two conditionally convergent series is conditionally convergent.
We still have to show that the series in \eqref{Fourier} does not converge absolutely if $\theta\ne \frac12\pi$.
According to~\cite[equation~(5.11.12)]{NIST:DLMF},
\begin{gather*}
\frac{\Ga(a+z)}{\Ga(b+z)}\sim z^{a-b}
\end{gather*}
as $z\to+\infty$.
Therefore, as $k\to\infty$,
\begin{gather*}
\frac{\Ga(\nu\!+\mu\!+k\!+1)}{\Ga\big(\nu+k+\frac32\big)}\frac{\big(\mu\!+\!\frac12\big)_k}{k!}
= \frac{\Ga(\nu\!+\mu\!+k\!+\!1)}{\Ga\big(\nu+k+\frac32\big)}
\frac{\Ga\big(\mu\!+\frac12\!+k\big)}{\Ga(k+1)\Ga\big(\mu+\frac12\big)}
\sim \frac{k^{\,\mu-\frac12} k^{\,\mu-\frac12}}{\Ga\big(\mu+\frac12\big)}=\frac{k^{2\mu-1}}{\Ga\big(\mu+\frac12\big)} .
\end{gather*}
Since $1/\Ga\big(\mu+\frac12\big)\ne 0$, there are positive constants $\kappa$ and $K$ such that
\begin{gather*} \left|\frac{\Ga(\nu+\mu+k+1)}{\Ga\big(\nu+k+\frac32\big)}\frac{\big(\mu+\frac12\big)_k}{k!}\right|\ge
\frac{\kappa}{k},
\end{gather*}
for $k\ge K$.
The second part of statement $(b)$ now follows from Lemma~\ref{l3}.

$(c)$ Suppose that $\mathop{\rm Re}\mu\ge \frac12$ and the series in \eqref{Fourier} converges. Then the terms of the series must converge to $0$.
It follows that $\cos((\nu+\mu+2k+1)\theta)\to 0$ as $k\to\infty$. By Lemma~\ref{l4} this is impossible unless $\theta=\frac12\pi$.
Therefore, the series in \eqref{Fourier} will diverge for $\theta\ne \frac12\pi$.
\end{proof}

\begin{Lemma}\label{l3}
Let $a\in \C$ and
$\theta\in(0,\pi)$.
Then
\begin{gather*}
\sum_{k=1}^\infty \frac1k |\cos((a+2k)\theta)|=\infty
\end{gather*}
unless $\cos a=0$ and $\theta=\frac12\pi$.
\end{Lemma}
\begin{proof}
If $\theta=\frac12\pi$ then $|\cos((a+2k)\theta)|$ is independent of $k$ and the assertion follows.
Now suppose that $\theta\ne \frac12\pi$.
Since $|\cos z|\ge |\cos(\mathop{\rm Re} z)|$ for all $z\in\C$ it is enough to consider $a\in\R$.
Then $|\cos((a+2k)\theta)|\ge \cos^2((a+2k)\theta)$.
Therefore, it is enough to show that
\begin{gather*}
\sum_{k=1}^\infty \frac1k \cos^2((a+2k)\theta) =\infty .
\end{gather*}
Now
\begin{gather*}
\frac1k\cos^2((a+2k)\theta) =\frac1{2k}+\frac1{2k} \cos(2(a+2k)\theta) .
\end{gather*}
The series $\sum_{k=1}^\infty 1/(2k)$ diverges. The series
$\sum_{k=1}^\infty\cos(2(a+2k)\theta)/(2k)$ converges by the Dirichlet test
provided the partial sums $\sum_{k=1}^n \cos(2(a+2k)\theta)$ form a bounded sequence.
This is true for~$\theta\in(0,\pi)$, $\theta\ne \frac12\pi$.
\end{proof}

\begin{Lemma}\label{l4}
Let $a,b\in\C$ and $\sin b\ne 0$. Then the sequence $\cos(a+bn)$ does not converge to $0$ as $n\to\infty$.
\end{Lemma}
\begin{proof}
Suppose that $\cos(a+bn)\to 0$ as $n\to\infty$.
We have
\begin{gather*}
\cos(a+b(n+1))=\cos(a+bn)\cos b-\sin(a+bn)\sin b.
\end{gather*}
If we let $n\to\infty$ and use $\sin b\ne 0$ we obtain that $\sin(a+bn)\to 0$. Since $\cos^2x+\sin^2x=1$ this is a contradiction.
\end{proof}

\section{{Convergence regions of the} Gauss hypergeometric series}\label{sec7}

In the hypergeometric representations of $\Q_\nu^\mu(x)$ derived in Sections \ref{sec3}--\ref{sec5} we worked with the principal value of
the hypergeometric function defined on $\C\setminus[1,\infty)$. If we wish to use the hypergeometric series~\eqref{Fseries}
we have to add the condition $|w_j|<1$.
Therefore, it is of interest to~determine the regions in the $x$-plane on which $|w_j|<1$. In~most cases these regions are obvious but not in all of them.

\subsection{Group I}
\begin{itemize}\itemsep=0pt
\item The hypergeometric functions in Theorem~\ref{tI1} have the argument {$w_1=(1-x)/2$}.
Now $|w_1| <1$ gives the condition $|1-x|<2$ describing a disk centered at $x=1$ with radius $2$.

\item The hypergeometric functions in Theorem~\ref{tI2} have the argument {$w_2=(1+x)/2$}.
Then $|w_2|<1$ is satisfied if and only if $|1+x|<2$.

\item The hypergeometric functions in Theorem~\ref{tI3} have the argument {$w_3=(x-1)/(x+1)$}.
Then $|w_3|<1$ is satisfied if and only if $\mathop{\rm Re} x>0$.

\item The hypergeometric functions in Theorem~\ref{tI4} have the argument {$w_4=(x+1)/(x-1)$}.
Then $|w_4|<1$ is satisfied if and only if $\mathop{\rm Re} x<0$.

\item The hypergeometric functions in Theorem~\ref{tI5} have the argument {$w_5=2/(1+x)$}.
Then $|w_5|<1$ is satisfied if and only if $|1+x|>2$.

\item The hypergeometric function in Theorem~\ref{tI6} have the argument {$w_6=2/(1-x)$}.
Then $|w_6|<1$ is satisfied if and only if $|1-x|>2$.
\end{itemize}

\subsection{Group II}
\begin{itemize}\itemsep=0pt
\item The hypergeometric functions in Theorem~\ref{tII1} have the argument $w_7=1-x^2$.
Then $|w_7|<1$ is satisfied if and only if $|1-x||1+x|<1$. This is the shaded region in Figure~\ref{fig2} bounded by a {lemniscate}.
\begin{figure}[ht]
 \centering
 \includegraphics[scale=.68]{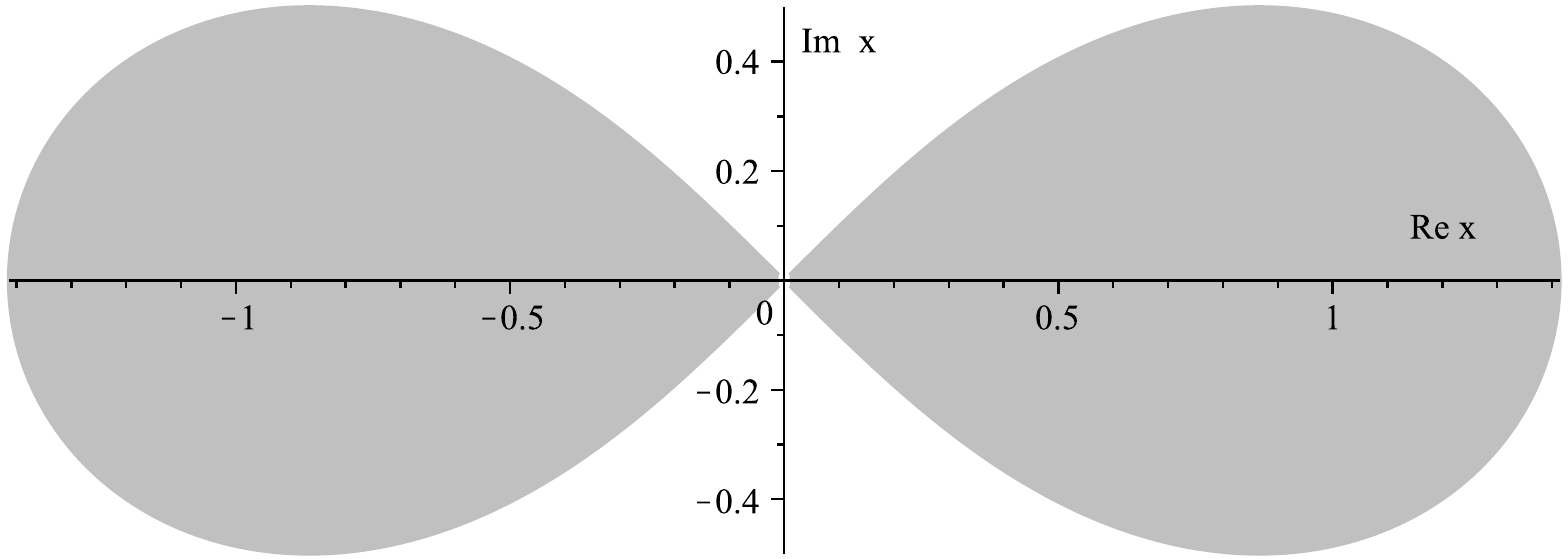}
 \caption{The region $|w_7|<1$.}\label{fig2}
\end{figure}

\item The hypergeometric functions in Theorem~\ref{tII2} have the argument {$w_8=1/\big(1-x^2\big)$}.
Then $|w_8|<1$ is satisfied if and only if $|1+x||1-x|>1$, that is, if $x$ lies in the unshaded region in Figure~\ref{fig2}.

\item The hypergeometric functions in Theorem~\ref{tII3} have the argument $w_9=x^2$.
Then $|w_9| <1$ is satisfied if and only if $|x|<1$.

\item The hypergeometric functions in Theorem~\ref{tII4} have the argument {$w_{10}=1/x^2$}.
Then $|w_{10}|<1$ is satisfied if and only if $|x|>1$.

\item The hypergeometric functions in Theorem~\ref{tII5} have the argument {$w_{11}=\big(x^2-1\big)/x^2$}.
Then $|w_{11}|<1$ is equivalent to $\mathop{\rm Re} x^2>\frac12$ or $(\mathop{\rm Re} x)^2-(\mathop{\rm Im} x)^2>\frac12$. This determines the shaded region
bounded by a hyperbola depicted in Figure~\ref{fig3}.

\begin{figure}[ht] \centering
 \includegraphics[scale=.63]{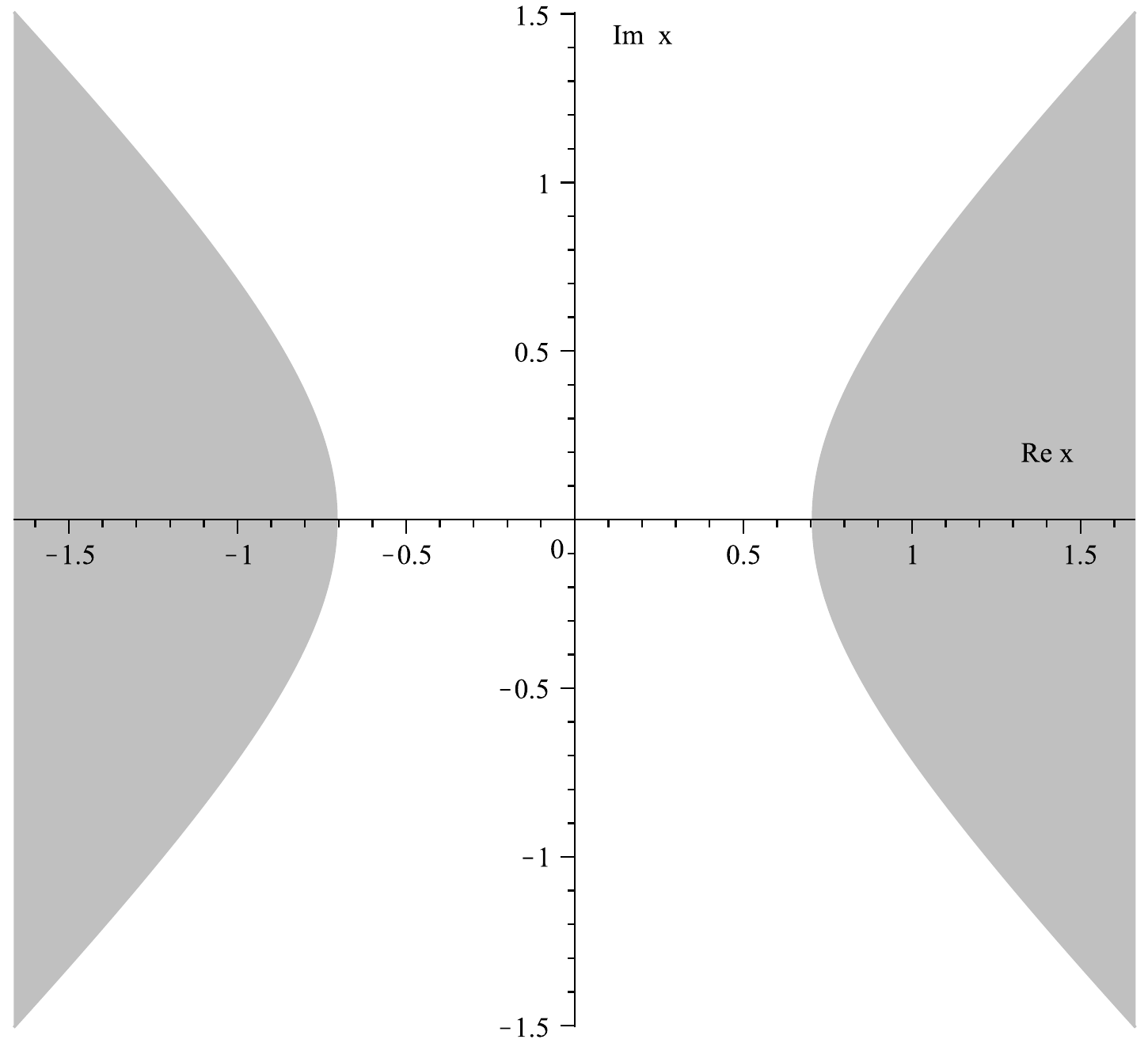}
 \caption{The region $|w_{11}|<1$.}\label{fig3}
\end{figure}

\item The hypergeometric functions in Theorem~\ref{tII6} have the argument {$w_{12}=x^2/\big(x^2-1\big)$}.
Then $|w_{12}|<1$ is satisfied if and only if $(\mathop{\rm Re} x)^2-(\mathop{\rm Im} x)^2<\frac12$. This determines the unshaded region in the $x$-plane
in Figure~\ref{fig3}.
\end{itemize}

\subsection{Group III}
For abbreviation, let
\begin{gather*}
y={\rm i}\sqrt{1-x^2}
\end{gather*}
for $x\in D_1$.
\begin{itemize}
\item The hypergeometric functions in Theorem~\ref{tIII1} (upper sign) have the argument
\begin{gather*}
w_{13}=\frac{-x+y}{2y}= \frac{w_{14}}{w_{14}-1},
\qquad \text{where}\quad
w_{14}=\frac{x-y}{x+y} .
\end{gather*}
Therefore, $|w_{13}|<1$ is equivalent to $\mathop{\rm Re} w_{14}<\frac12$. Let $x\in D_1$. Then $x=\cos\theta$, $\theta=\alpha+ {\rm i}\beta$, where $\alpha\in(0,\pi)$ and $\beta\in\R$, so
\begin{gather*}
w_{14}=\frac{\cos\theta-{\rm i}\sin\theta}{\cos\theta+{\rm i}\sin\theta}= \expe^{-2{\rm i}\theta}= \expe^{-2{\rm i}\alpha} \expe^{2\beta} .
\end{gather*}
Hence, $|w_{13}|{<}1$ is equivalent to
${\expe^{2\beta}\cos(2\alpha){<}\frac12}$.
Then $\frac14\pi\le \alpha\le \frac34\pi$ or
$\beta{<}-\frac12 \ln(2\cos(2\alpha))$. In~the $x$-plane the curve
$\beta=-\frac12 \ln(2\cos(2\alpha))$,
for $0<\alpha<\frac14\pi$
is given by
\begin{gather}\label{curve}
 x=\cos(\alpha+ {\rm i}\beta)=\frac12\big(t+t^{-1}\big)\cos\alpha+\frac12{\rm i}\big(t-t^{-1}\big)\sin\alpha,
\end{gather}
where $t=\sqrt{2\cos(2\alpha)}$.
This curve and the corresponding curve for $\frac34\pi<\alpha<\pi$ are shown in Figure~\ref{fig4}.
Therefore, for $x\in D_1$, $|w_{13}|<1$ is satisfied if and only if $x$ lies in the unshaded region in Figure~\ref{fig4}.
\begin{figure}[h!]
 \centering
 \includegraphics[scale=.63]{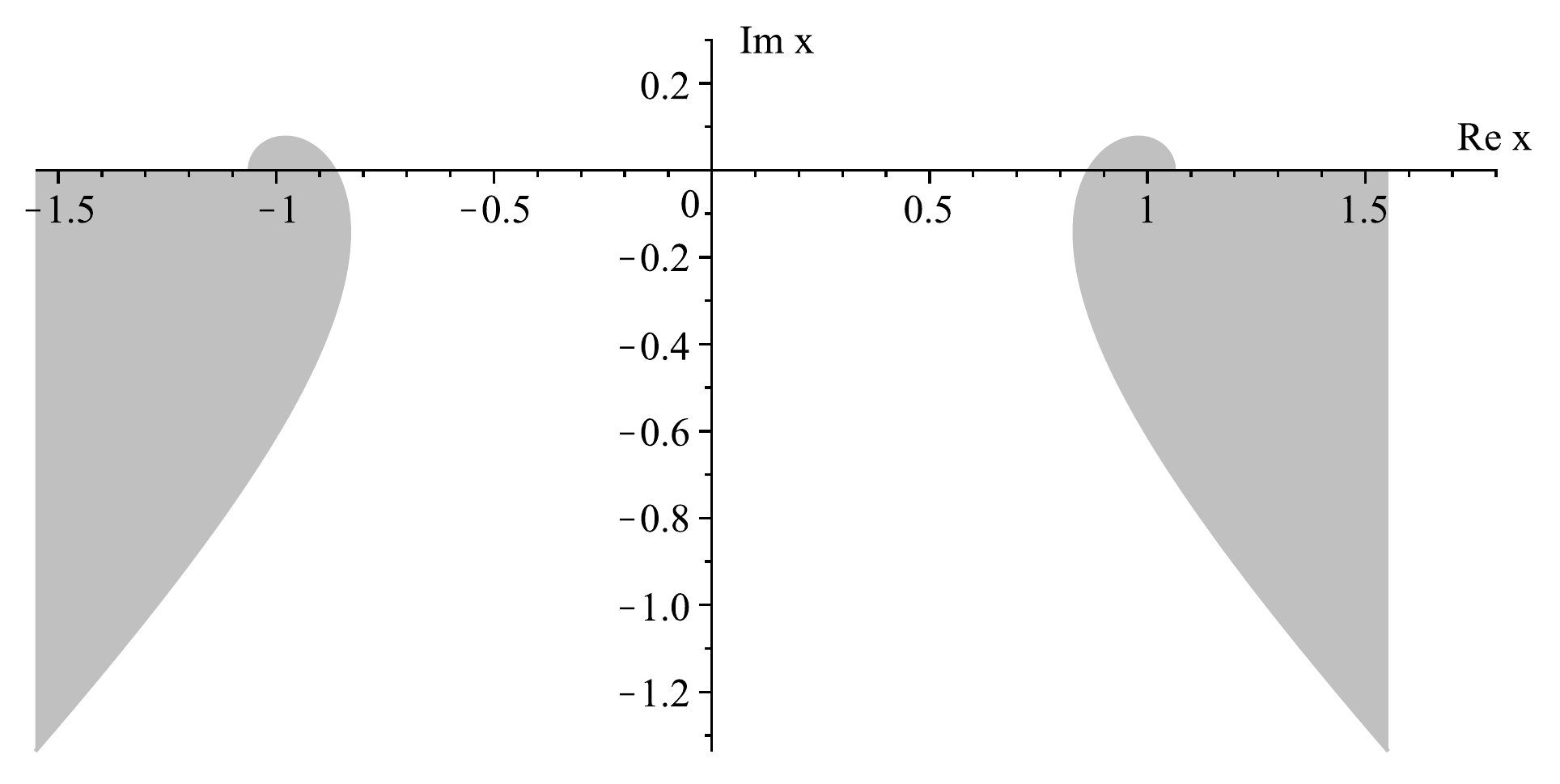}
 \caption{The region $|w_{16}|<1$ ($|w_{13}|>1$).}\label{fig4}
\end{figure}

\item The hypergeometric functions in Theorem~\ref{tIII1} (lower sign) have the argument
\begin{gather*}
w_{17}=\frac{x+y}{2y}.
\end{gather*}
The condition $|w_{17}|\!<\!1$ is satisfied if and only if $x$ lies in the unshaded region of Figure~\ref{fig5}.
\begin{figure}[h!]
 \centering
 \includegraphics[scale=.63]{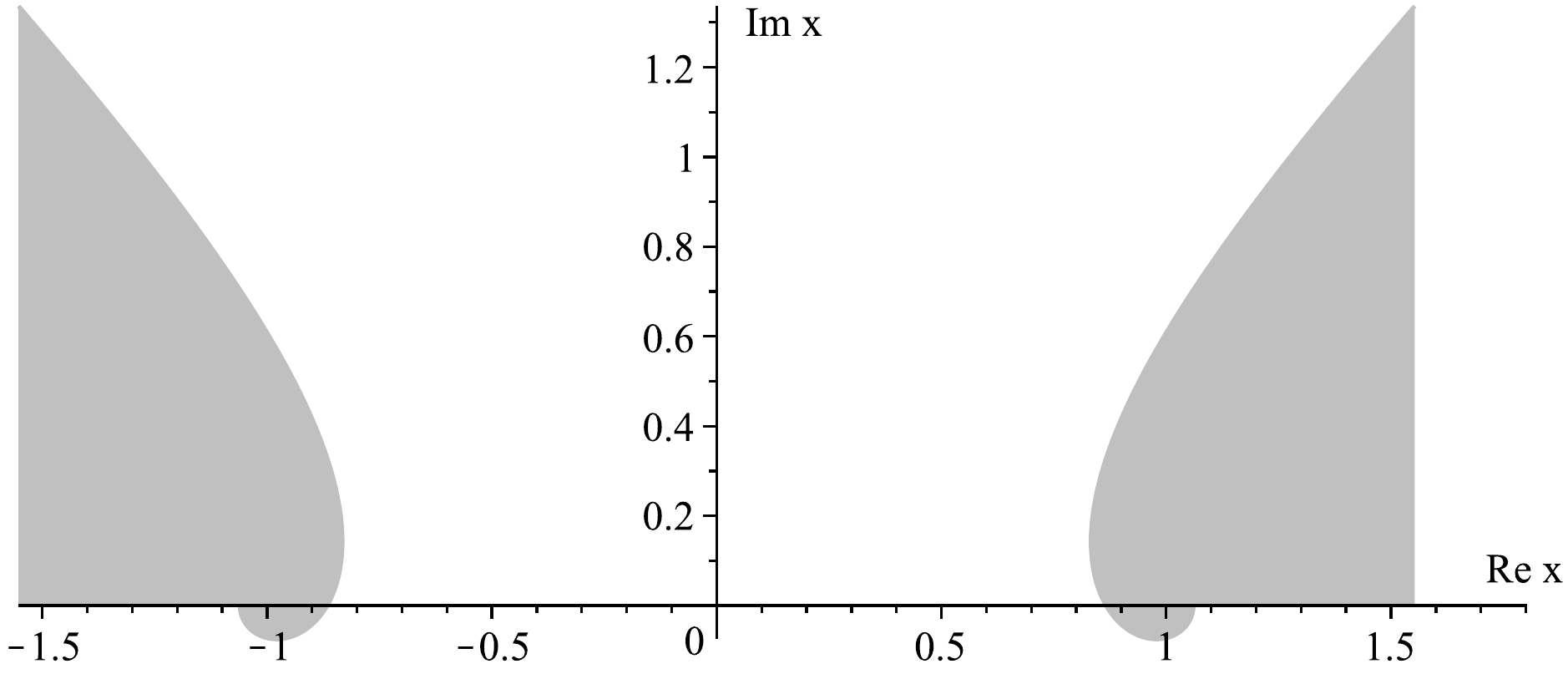}
 \caption{The region
$|w_{15}|<1$ ($|w_{17}|>1$).}\label{fig5}
\end{figure}

\item The hypergeometric functions in Theorem~\ref{tIII2} (upper sign) have the argument
\begin{gather*}
w_{14}=\frac{x-y}{x+y}.
\end{gather*}
We take $\sqrt{x^2-1}={\rm i}\sqrt{1-x^2}$ for $x\in D_1$.
Then using $x=\cos\theta$, $\mathop{\rm Re}\theta\in(0,\pi)$ gives $w_{14}=\expe^{-2{\rm i}\theta}$
so $|w_{14}|=\expe^{2\mathop{\rm Im}\theta}$. Then $|w_{14}|<1$ means $\mathop{\rm Im}\theta<0$ which therefore implies that
$\mathop{\rm Im} x>0$.
Then $|w_{14}|=\expe^{2\beta}$ so $|w_{14}|<1$ is satisfied if and only if $\mathop{\rm Im} x>0$.

\item The hypergeometric functions in Theorem~\ref{tIII2} (lower sign) have the argument
\begin{gather*}
w_{18}=\frac{x+y}{x-y} .
\end{gather*}
Since $w_{18}=1/w_{14}$,
then $|w_{18}|<1$ if and only if $\mathop{\rm Im} x<0$.
Of course, these
results depend on the choice of the sign of the root $\sqrt{x^2-1}$.

\item The hypergeometric functions in Theorem~\ref{tIII3} (upper sign) have the argument
\begin{gather*}
w_{15}=\frac{2y}{x+y}=\frac1{w_{17}} .
\end{gather*}
Then $|w_{15}|<1$ is equivalent to $|w_{17}|>1$ and this means that $x$ lies in the shaded region of Figure~\ref{fig5}.

\item The hypergeometric functions in Theorem~\ref{tIII3} (lower sign) have the argument
\begin{gather*}
w_{16}=\frac{2y}{-x+y}=\frac{1}{w_{13}} .
\end{gather*}
Then $|w_{16}|<1$ is equivalent to $|w_{13}|>1$ and this means that $x$ lies in the shaded region of Figure~\ref{fig5}.

We close this paper with the following interesting observations. In~the arguments $w_j$, $j=13,\dots,18$, we
replaced $\sqrt{x^2-1}$ by ${\rm i}\sqrt{1-x^2}$ because Ferrers functions are naturally defined on the domain $D_1$.
However, if we are interested in hypergeometric representations of the associated Legendre function $Q_\nu^\mu(x)$
we will use $\sqrt{x^2-1}:=\sqrt{x-1}\sqrt{x+1}$ which changes the domains on which $|w_j|<1$.
Let us denote these modifications of $w_j$ by $w_j^\ast$.
Since $w_{13}(x)=w_{13}^\ast(x)$ for $\mathop{\rm Im} x>0$ the part of the curve $|w^\ast_{13}|=1$ lying in the region $\mathop{\rm Re} x>0$, $\mathop{\rm Im} x>0$ is given by~\eqref{curve} for $0<\alpha<\frac16\pi$.
The parts of the curve
in the other quadrants are reflections of this arc at the real and imaginary axis.
The curve surrounding~$1$ is
shown in Figure~\ref{fig6}.
This means that the hypergeometric representation~\eqref{MOS31} with the Gauss hypergeometric series in place of {${}_2F_1$}
holds for all~$z$ that lie outside these curves.
Moreover, we obtain that $|w^\ast_{14}|<1$ for all $x\in D_2$. Therefore, the hypergeometric representation~\eqref{MOS32} with the Gauss hypergeometric series in place of
{${}_2F_1$} holds for all $z\in D_2$.
This appears to be the only representation of $Q_\nu^\mu$ valid on $D_2$ by the Gauss hypergeometric series.
\begin{figure}[h!]
 \centering
 \includegraphics[scale=.63]{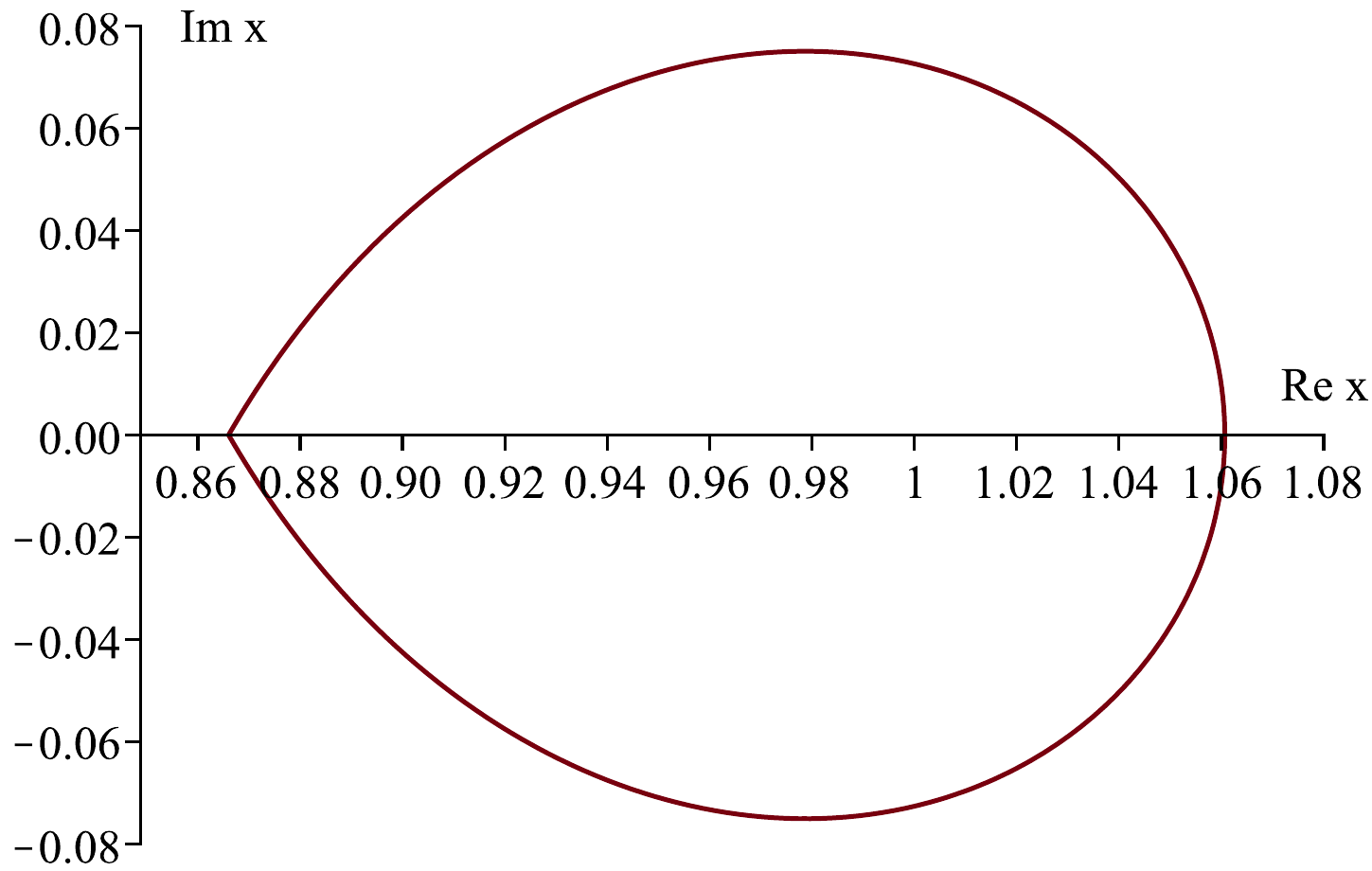}
 \caption{The curve $|w^\ast_{13}|=1$.
 Note that by analyzing the equation of the curve one can see that the innermost points occur at
 $x=\pm\frac{\sqrt{3}}{2}\approx \pm 0.866025$ and the outermost points occur at $x=\pm\frac{3\sqrt{2}}{4}\approx\pm 1.06066$.
 Also, the points with the largest absolute value of their imaginary parts on the curve occur at~$\mathop{\rm Re} x=\pm\sqrt{\frac{15}{32}+\frac{7\sqrt{5}}{32}}\approx \pm 0.978718$
 and $\mathop{\rm Im} x=\pm \sqrt{-\frac{11}{32}+\frac{5\sqrt{5}}{32}}\approx \pm 0.0750708$.}\label{fig6}
\end{figure}
\end{itemize}

\appendix

\section{Olbricht's 72 solutions of the associated Legendre equation}
\label{Olbrichtseventytwo}

Richard Olbricht~\cite[{pp.~17--20}]{Olbricht1888} gave a list of 72 solutions of the associated Legendre equation~\eqref{Legendre}
in terms of the Gauss hypergeometric function ${}_2F_1(a,b;c;z)$.
He does not state the domains of definition of these solutions and does not identify them in terms of known solutions.
Therefore, it might be useful to go through the list and add a few remarks to each entry.

\medskip

We will regularly use Euler's and Pfaff's transformations~\eqref{Euler}--\eqref{Pfaff2}
and the domains~\eqref{D1}--\eqref{D3}.
Olbricht's solutions will be denoted by $\L{{\rm I}}{j}$, $\L{{\rm II}}{j}$, $\L{{\rm III}}{j}$, $j=1,2,\dots,24$.

\subsection{The Group I arguments}
Consider
\begin{gather}\label{1}
\L{{\rm I}}{1}(x):=\left(\frac{1-x}{2}\right)^{\frac12\mu}\left(\frac{1+x}{2}\right)^{\frac12\mu} {}_2F_1\left(\mu-\nu,\nu+\mu+1;1+\mu;\frac{1-x}{2}\right).
\end{gather}
If we use principal values of the powers and also of the hypergeometric function {${}_2F_1$} then
$\L{{\rm I}}{1}$ is analytic on the domain $D_1$.
Using \eqref{Euler} and~\cite[equation~(14.3.1)]{NIST:DLMF} Olbricht has
\begin{gather*}
\L{{\rm I}}{1}(x)=\Gamma(1+\mu) \P_\nu^{-\mu}(x)
\end{gather*}
for $x\in D_1$.
If $\mathbf{L}^I_1$ denotes the right-hand side of \eqref{1} with {${}_2F_1$} replaced
by Olver's $\FF(a,b;c;z):=\frac{1}{\Gamma(c)}{{}_2F_1}(a,b;c;z)$
\cite[equation~(15.2.2)]{NIST:DLMF}
then
\begin{gather*}
\mathbf{L}^I_1(x)=\P_\nu^{-\mu}(x)
\end{gather*}
and this equation is true for all $\nu,\mu\in\C$.
The function
\begin{gather*}
\L{{\rm I}}{2}(x):=\left(\frac{1-x}{2}\right)^{-\frac12\mu}\left(\frac{1+x}{2}\right)^{\frac12\mu} {}_2F_1\left({-}\nu,\nu+1;1-\mu;\frac{1-x}{2}\right)
\end{gather*}
is analytic on $D_1$.
By applying \eqref{Euler} to $\L{{\rm I}}{1}$ and changing $\mu$ to $-\mu$,
\begin{gather*}
\L{{\rm I}}{2}(x)=\Gamma(1-\mu) \P_\nu^{\mu}(x),
\end{gather*}
for $x\in D_1$. Olbricht has
\begin{gather*}
\L{{\rm I}}{3}(x):=\left(\frac{1-x}{2}\right)^{\frac12\mu}\left(\frac{1+x}{2}\right)^{-\frac12\mu} {}_2F_1\left({-}\nu,\nu+1;1+\mu;\frac{1-x}{2}\right),
\end{gather*}
and by~\eqref{Euler}, $\L{{\rm I}}{1}=\L{{\rm I}}{3}$. Also,
\begin{gather*}
\L{{\rm I}}{4}(x):=\left(\frac{1-x}{2}\right)^{-\frac12\mu}\left(\frac{1+x}{2}\right)^{-\frac12\mu} {}_2F_1\left({-}\nu-\mu,\nu-\mu+1;1-\mu;\frac{1-x}{2}\right),
\end{gather*}
and by~\eqref{Euler}, $\L{{\rm I}}{2}=\L{{\rm I}}{4}$. One also has
\begin{gather*}
\L{{\rm I}}{5}(x):=\L{{\rm I}}{1}(-x),
\\
\L{{\rm I}}{6}(x):=\L{{\rm I}}{2}(-x),
\\
\L{{\rm I}}{7}(x):=\L{{\rm I}}{3}(-x),
\\
\L{{\rm I}}{8}(x):=\L{{\rm I}}{4}(-x),
\\
\L{{\rm I}}{9}(x):=\left(\frac{2}{1-x}\right)^{-\nu}\left(\frac{x+1}{x-1}\right)^{\frac12\mu}
{}_2F_1\left({-}\nu,\mu-\nu;-2\nu;\frac{2}{1-x}\right).
\end{gather*}
Using principal values, $\L{{\rm I}}{9}$ is analytic on $D_3$.
Using~\cite[Entry 24, p.~161]{MOS} and~\eqref{Euler} with~\cite[equation~(14.9.14)]{NIST:DLMF}
\begin{gather*}
\QQ_\nu^\mu(x)=\QQ_\nu^{-\mu}(x)
\end{gather*}
and~\cite[equation~(14.3.10)]{NIST:DLMF}
\begin{gather*}
\QQ_\nu^\mu(x):=\expe^{-\mu\pi {\rm i}}\frac{Q_\nu^\mu(x)}{\Gamma(\nu+\mu+1)},
\end{gather*}
we find
\begin{gather*}
\L{{\rm I}}{9}(x)= \frac{4^{-\nu}}{\sqrt\pi}\Gamma\left(\frac12-\nu\right) \QQ_{-\nu-1}^\mu(-x)
\end{gather*}
for $x\in D_3$. We also have
\begin{gather*}
\L{{\rm I}}{10}(x):=\left(\frac{2}{1-x}\right)^{\nu+1}\left(\frac{x+1}{x-1}\right)^{\frac12\mu} {}_2F_1\left(\nu+\mu+1,\nu+1;2\nu+2;\frac{2}{1-x}\right),
\end{gather*}
{and} $\L{{\rm I}}{10}$ is obtained from $\L{{\rm I}}{9}$ by replacing $\nu \mapsto -\nu-1$. Hence
\begin{gather*}
\L{{\rm I}}{10}(x)= \frac{4^{\nu+1}}{\sqrt\pi}\Gamma\left(\nu+\frac32\right) \QQ_\nu^\mu(-x)
\end{gather*}
for $x\in D_3$. Also,
\begin{gather*}
\L{{\rm I}}{11}(x):=\left(\frac{2}{1-x}\right)^{-\nu}\left(\frac{x+1}{x-1}\right)^{-\frac12\mu}
{}_2F_1\left({-}\nu,-\nu-\mu;-2\nu;\frac{2}{1-x}\right),
\end{gather*}
and $\L{{\rm I}}{11}$ is obtained from $\L{{\rm I}}{9}$ by replacing $\mu\mapsto -\mu$. Hence $\L{{\rm I}}{11}=\L{{\rm I}}{9}$.
One also has
\begin{gather*}
\L{{\rm I}}{12}(x):=\left(\frac{2}{1-x}\right)^{\nu+1}\left(\frac{x+1}{x-1}\right)^{-\frac12\mu}
{}_2F_1\left(\nu+1,\nu-\mu+1;2\nu+2;\frac{2}{1-x}\right),
\end{gather*}
where $\L{{\rm I}}{12}$ is obtained from $\L{{\rm I}}{10}$ by replacing $\mu\mapsto -\mu$, so $\L{{\rm I}}{12}=\L{{\rm I}}{10}$.
Also,
\begin{gather*}
\L{{\rm I}}{13}(x):=\L{{\rm I}}{9}(-x),
\end{gather*}
and using principal values, $\L{{\rm I}}{13}(x)$ is analytic on $D_2$, and therefore
\begin{gather*}
\L{{\rm I}}{13}(x)= \frac{4^{-\nu}}{\sqrt\pi}\Gamma\left(\frac12-\nu\right) \QQ_{-\nu-1}^\mu(x)
\end{gather*}
for $x\in D_2$. For $x\in D_2$, we also have
\begin{gather*}\allowdisplaybreaks
\L{{\rm I}}{14}(x):=\L{{\rm I}}{10}(-x),
\\
\L{{\rm I}}{15}(x):=\L{{\rm I}}{11}(-x),
\\
\L{{\rm I}}{16}(x):=\L{{\rm I}}{12}(-x),
\\
\L{{\rm I}}{17}(x):=\left(\frac{x-1}{x+1}\right)^{\frac12\mu}\left(\frac2{1+x}\right)^{-\nu} {}_2F_1\left({-}\nu,\mu-\nu;1+\mu;\frac{x-1}{x+1}\right).
\end{gather*}
This function is analytic on $D_2$.
By~\cite[Section~4.1.2, Entry 3]{MOS}
\begin{gather*}
\L{{\rm I}}{17}(x)=\Gamma(1+\mu)P_\nu^{-\mu}(x)
\end{gather*}
for $x\in D_2$. One also has
\begin{gather*}
\L{{\rm I}}{18}(x):=\left(\frac{x-1}{x+1}\right)^{-\frac12\mu}\left(\frac2{1+x}\right)^{-\nu} {}_2F_1\left({-}\nu,-\mu-\nu;1-\mu;\frac{x-1}{x+1}\right),
\end{gather*}
where $\L{{\rm I}}{18}$ is obtained from $\L{{\rm I}}{17}$ by replacing $\mu\mapsto -\mu$. Hence
\begin{gather*}
\L{{\rm I}}{18}(x):=\Gamma(1-\mu)P_\nu^{\mu}(x)
\end{gather*}
for $x\in D_2$. Also,
\begin{gather*}
\L{{\rm I}}{19}(x):=\left(\frac{x-1}{x+1}\right)^{\frac12\mu}\left(\frac2{1-x}\right)^{\nu+1} {}_2F_1\left(\nu+\mu+1,\nu+1;1+\mu;\frac{x-1}{x+1}\right),
\end{gather*}
and by \eqref{Euler}, $\L{{\rm I}}{19}=\L{{\rm I}}{17}$. Also,
\begin{gather*}
\L{{\rm I}}{20}(x):=\left(\frac{x-1}{x+1}\right)^{-\frac12\mu}\left(\frac2{1+x}\right)^{\nu+1} {}_2F_1\left(\nu+1,\nu-\mu+1;1-\mu;\frac{x-1}{x+1}\right),
\end{gather*}
and by \eqref{Euler}, $\L{{\rm I}}{20}=\L{{\rm I}}{18}$. Also,
\begin{gather*}
\L{{\rm I}}{21}(x):=\L{{\rm I}}{17}(-x),
\\
\L{{\rm I}}{22}(x):=\L{{\rm I}}{18}(-x),
\\
\L{{\rm I}}{23}(x):=\L{{\rm I}}{19}(-x),
\\
\L{{\rm I}}{24}(x):=\L{{\rm I}}{21}(-x).
\end{gather*}

\subsection{The Group II arguments}

The first function in Group II of Olbricht's list is
\begin{gather*}
\L{{\rm II}}{1}(x):=\big(1-x^2\big)^{\frac12\mu} {}_2F_1\left(\frac12\mu-\frac12\nu,\frac12\nu+\frac12\mu+\frac12;\frac12;x^2\right).
\end{gather*}
It is an even solution of \eqref{Legendre} on $D_1$ uniquely determined by the initial conditions
$y(0)=1$, $y'(0)=0$. By~\cite[equation~(14.5.1)]{NIST:DLMF},
\begin{gather*}
\L{{\rm II}}{1}(x)=2^{-\mu-1}\pi^{{-\frac12}}\Gamma\left(\frac12\nu-\frac12\mu+1\right)
\Gamma\left({-}\frac12\nu-\frac12\mu+\frac12\right)\big(\P_\nu^\mu(x)+\P_\nu^\mu(-x)\big)
\end{gather*}
for $x\in D_1$. We also have
\begin{gather*}
\L{{\rm II}}{2}(x):=x\big(1-x^2\big)^{\frac12\mu} {}_2F_1\left(\frac12\mu-\frac12\nu+\frac12,\frac12\nu+\frac12\mu+1;\frac32;x^2\right).
\end{gather*}
This function is analytic on $D_1$. It is an odd solution of \eqref{Legendre} on $D_1$ uniquely determined by the initial conditions
$y(0)=0$, $y'(0)=1$. Using~\cite[equation~(14.5.2)]{NIST:DLMF},
\begin{gather*}
\L{{\rm II}}{2}(x):=2^{-\mu-2}\pi^{{-\frac12}}\Gamma\left(\frac12\nu-\frac12\mu+\frac12\right) \Gamma\left({-}\frac12\nu-\frac12\mu\right)\big(\P_\nu^\mu(-x)-\P_\nu^\mu(x)\big),
\end{gather*}
for $x\in D_1$. Also,
\begin{gather*}
\L{{\rm II}}{3}(x):=\big(1-x^2\big)^{-\frac12\mu} {}_2F_1\left({-}\frac12\nu-\frac12\mu,\frac12\nu-\frac12\mu+\frac12;\frac12;x^2\right),
\end{gather*}
and by \eqref{Euler}, $\L{{\rm II}}{3}=\L{{\rm II}}{1}$.
Furthermore,
\begin{gather*}
\L{{\rm II}}{4}(x):=x\big(1-x^2\big)^{-\frac12\mu} {}_2F_1\left({-}\frac12\nu-\frac12\mu+\frac12,\frac12\nu-\frac12\mu+1;\frac32;x^2\right),
\end{gather*}
and by \eqref{Euler}, $\L{{\rm II}}{4}=\L{{\rm II}}{2}$.
Also,
\begin{gather*}
\L{{\rm II}}{5}(x):=\big(1-x^2\big)^{\frac12\mu}{}_2F_1\left(\frac12\mu-\frac12\nu,\frac12\nu+\frac12\mu+\frac12;1+\mu;1-x^2\right).
\end{gather*}
This function is analytic on $D_1\setminus {\rm i}\R$. By~\cite[Section~4.1.2, Entry 5]{MOS},
\begin{gather*}
\L{{\rm II}}{5}(x)=2^\mu\Gamma(1+\mu) \P_\nu^{-\mu}(x),
\end{gather*}
for $x\in D_1^+$. One also has
\begin{gather*}
\L{{\rm II}}{6}(x):=x\big(1-x^2\big)^{\frac12\mu}{}_2F_1\left(\frac12\mu-\frac12\nu+\frac12,\frac12\nu+\frac12\mu+1;1+\mu;1-x^2\right),
\end{gather*}
and by \eqref{Euler}, $\L{{\rm II}}{6}=\L{{\rm II}}{5}$ for $x\in D_1^+$. Also,
\begin{gather*}
\L{{\rm II}}{7}(x):=\big(1-x^2\big)^{-\frac12\mu}{}_2F_1\left({-}\frac12\nu-\frac12\mu,\frac12\nu-\frac12\mu+\frac12;1-\mu;1-x^2\right),
\end{gather*}
and $\L{{\rm II}}{7}$ is obtained from $\L{{\rm II}}{5}$ by replacing $\mu\mapsto -\mu$.
Therefore,
\begin{gather*}
\L{{\rm II}}{7}(x)=2^{-\mu}\Gamma(1-\mu) \P_\nu^\mu(x)
\end{gather*}
for $x\in D_1^+$. Also,
\begin{gather*}
\L{{\rm II}}{8}(x):=x\big(1-x^2\big)^{-\frac12\mu}{}_2F_1\left(\frac12\nu-\frac12\mu+1,-\frac12\nu-\frac12\mu+\frac12;1-\mu;1-x^2\right),
\end{gather*}
and by \eqref{Euler}, $\L{{\rm II}}{8}(x)=\L{{\rm II}}{7}(x)$ for $x\in D_1^+$. Also,
\begin{gather*}
\L{{\rm II}}{9}(x):=
{x^{\nu-\mu}}\big(x^2-1\big)^{\frac12\mu} {}_2F_1\left(\frac12\mu-\frac12\nu,\frac12\mu-\frac12\nu+\frac12;\frac12-\nu;\frac1{x^2}\right).
\end{gather*}
We use $\big(x^2-1\big)^\alpha$ as an abbreviation for $(x-1)^\alpha(x+1)^\alpha$, so $\big(x^2-1\big)^\alpha$ is analytic on $D_2$.
Olbricht has $\big(1-x^2\big)^{\frac12\mu}$ in place of $\big(x^2-1\big)^{\frac12\mu}$ but this has the disadvantage that when using
principal values the function is not analytic for any real $x$.
$\L{{\rm II}}{9}$ is analytic on $D_2$. By~\cite[Section~4.1.2, Entry 28]{MOS},
\begin{gather*}
\L{{\rm II}}{9}(x)=2^{-\nu}\pi^{{-\frac12}}\Gamma\left(\frac12-\nu\right)
\QQ_{-\nu-1}^\mu(x)
\end{gather*}
for $x\in D_2$. Also,
\begin{gather*}
\L{{\rm II}}{10}(x):=\left(\frac1x\right)^{\nu+\mu+1}\big(x^2-1\big)^{\frac12\mu} {}_2F_1\left(\frac12\nu+\frac12\mu+\frac12,\frac12\nu+\frac12\mu+1;\nu+\frac32;\frac1{x^2}\right).
\end{gather*}
Olbricht has $\big(1-x^2\big)^{\frac12\mu}$ in place of $\big(x^2-1\big)^{\frac12\mu}${, and}
$\L{{\rm II}}{10}$ is obtained from $\L{{\rm II}}{9}$ by replacing $\nu\mapsto -\nu-1$. Hence
\begin{gather*}
\L{{\rm II}}{10}(x)=2^{\nu+1}\pi^{{-\frac12}}\Gamma\left(\nu+\frac32\right)
\QQ_\nu^\mu(x)
\end{gather*}
for $x\in D_2$. Also,
 \begin{gather*}
\L{{\rm II}}{11}(x):={x^{{{\mu-\nu-1}}}}\big(x^2-1\big)^{-\frac12\mu} {}_2F_1\left(\frac12\nu-\frac12\mu+\frac12,\frac12\nu-\frac12\mu+1;\nu+\frac32;\frac1{x^2}\right).
\end{gather*}
Olbricht has $\big(1-x^2\big)^{\frac12\mu}$ in place of $\big(x^2-1\big)^{\frac12\mu}${, and}
$\L{{\rm II}}{11}$ is obtained from $\L{{\rm II}}{10}$ by replacing $\mu\mapsto -\mu$. Hence
$\L{{\rm II}}{11}=\L{{\rm II}}{10}$. Also,
\begin{gather*}
\L{{\rm II}}{12}(x):={x^{\nu+\mu}}\big(x^2-1\big)^{-\frac12\mu} {}_2F_1\left({-}\frac12\nu-\frac12\mu,-\frac12\nu-\frac12\mu+\frac12;\frac12-\nu;\frac1{x^2}\right).
\end{gather*}
Olbricht has $\big(1-x^2\big)^{\frac12\mu}$ in place of $\big(x^2-1\big)^{\frac12\mu}${, and}
$\L{{\rm II}}{12}$ is obtained from $\L{{\rm II}}{9}$ by replacing $\mu\mapsto -\mu$. Hence
$\L{{\rm II}}{12}=\L{{\rm II}}{9}$. Also,
\begin{gather*}
\L{{\rm II}}{13}(x):=\big(x^2-1\big)^{\frac12\nu} {}_2F_1\left(\frac12\mu-\frac12\nu,-\frac12\mu-\frac12\nu;\frac12-\nu;\frac{1}{1-x^2}\right).
\end{gather*}
Olbricht has $\big(1-x^2\big)^{\frac12\nu}$ in place of $\big(x^2-1\big)^{\frac12\nu}${, and} $\L{{\rm II}}{13}$ is analytic on $D_2$.
By \eqref{Pfaff1}, $\L{{\rm II}}{13}=\L{{\rm II}}{9}$. More precisely, we use \eqref{Pfaff1} for $x>1$ and then apply the identity theorem for analytic functions.
Also,
\begin{gather*}
\L{{\rm II}}{14}(x):=\big(x^2-1\big)^{-\frac12\nu-\frac12} {}_2F_1\left(\frac12\nu+\frac12\mu+\frac12,\frac12\nu-\frac12\mu+\frac12;\nu+\frac32;\frac{1}{1-x^2}\right).
\end{gather*}
Olbricht has $\big(1-x^2\big)^{-\frac12\nu-\frac12}$ in place of $\big(x^2-1\big)^{-\frac12\nu-\frac12}$, and by \eqref{Pfaff1}, $\L{{\rm II}}{14}=\L{{\rm II}}{10}$.
Also,
\begin{gather*}
\L{{\rm II}}{15}(x):=x\big(x^2-1\big)^{\frac12\nu-\frac12} {}_2F_1\left(\frac12\mu-\frac12\nu+\frac12,-\frac12\nu-\frac12\mu+\frac12;\frac12-\nu;\frac{1}{1-x^2}\right).
\end{gather*}
Olbricht has $\big(1-x^2\big)^{\frac12\nu-\frac12}$ in place of $\big(x^2-1\big)^{\frac12\nu-\frac12}${, and b}%
y \eqref{Pfaff1}, $\L{{\rm II}}{15}=\L{{\rm II}}{12}$.
Also,
\begin{gather*}
\L{{\rm II}}{16}(x):=x\big(x^2-1\big)^{\frac12\nu-1} {}_2F_1\left(\frac12\nu+\frac12\mu+1,\frac12\nu-\frac12\mu+1;\nu+\frac32;\frac{1}{1-x^2}\right).
\end{gather*}
Olbricht has $\big(1-x^2\big)^{\frac12\nu-1}$ in place of $\big(x^2-1\big)^{\frac12\nu-1}$, and by \eqref{Pfaff1}, $\L{{\rm II}}{16}=\L{{\rm II}}{11}$.
Also,
\begin{gather*}
\L{{\rm II}}{17}(x):=\big(1-x^2\big)^{\frac12\nu}{}_2F_1\left(\frac12\mu-\frac12\nu,-\frac12\mu-\frac12\nu; \frac12;\frac{x^2}{x^2-1}\right),
\end{gather*}
and by \eqref{Pfaff1}, $\L{{\rm II}}{17}=\L{{\rm II}}{1}$.
Also,
\begin{gather*}
\L{{\rm II}}{18}(x):=\big(1-x^2\big)^{-\frac12\nu-\frac12} {}_2F_1\left(\frac12\nu+\frac12\mu+\frac12,\frac12\nu-\frac12\mu+\frac12;\frac12; \frac{x^2}{x^2-1}\right),
\end{gather*}
and by \eqref{Pfaff1}, $\L{{\rm II}}{18}=\L{{\rm II}}{3}$.
Also,
\begin{gather*}
\L{{\rm II}}{19}(x):=x\big(1-x^2\big)^{\frac12\nu-\frac12} {}_2F_1\left(\frac12\mu-\frac12\nu+\frac12,-\frac12\nu-\frac12\mu+\frac12;\frac32; \frac{x^2}{x^2-1}\right),
\end{gather*}
and by \eqref{Pfaff1}, $\L{{\rm II}}{19}=\L{{\rm II}}{2}$.
Also,
\begin{gather*}
\L{{\rm II}}{20}(x):=x\big(1-x^2\big)^{-\frac12\nu-1} {}_2F_1\left(\frac12\nu+\frac12\mu+1,\frac12\nu-\frac12\mu+1;\frac32;\frac{x^2}{x^2-1}\right),
\end{gather*}
and by \eqref{Pfaff1}, $\L{{\rm II}}{20}=\L{{\rm II}}{4}$.
Also,
\begin{gather*}
\L{{\rm II}}{21}(x):=x^{{\nu-\mu}}\big(1-x^2\big)^{\frac12\mu}{}_2F_1 \left(\frac12\mu-\frac12\nu,\frac12\mu-\frac12\nu+\frac12;1+\mu;
\frac{x^2-1}{x^2}\right),
\end{gather*}
and by \eqref{Pfaff1}, $\L{{\rm II}}{21}=\L{{\rm II}}{5}$.
Also,
\begin{gather*}
\L{{\rm II}}{22}(x):=x^{{-\nu-\mu-1}}\big(1-x^2\big)^{\frac12\mu} {}_2F_1\left(\frac12\nu+\frac12\mu+\frac12,\frac12\nu+\frac12\mu+1;1+\mu;\frac{x^2-1}{x^2}\right),
\end{gather*}
and by \eqref{Pfaff1}, $\L{{\rm II}}{22}=\L{{\rm II}}{6}$.
Also,
\begin{gather*}
\L{{\rm II}}{23}(x):=x^{{\mu-\nu-1}}\big(1-x^2\big)^{-\frac12\mu} {}_2F_1\left(\frac12\nu-\frac12\mu+1,\frac12\nu-\frac12\mu+\frac12;1-\mu;\frac{x^2-1}{x^2}\right),
\end{gather*}
and by \eqref{Pfaff1}, $\L{{\rm II}}{23}=\L{{\rm II}}{7}$.
Also,
\begin{gather*}
\L{{\rm II}}{24}(x):=x^{\mu+\nu}\big(1-x^2\big)^{-\frac12\mu} {}_2F_1\left({-}\frac12\nu-\frac12\mu+\frac12,-\frac12\mu-\frac12\nu;1-\mu;\frac{x^2-1}{x^2}\right),
\end{gather*}
and by \eqref{Pfaff1}, $\L{{\rm II}}{24}=\L{{\rm II}}{8}$.

\subsection{The Group III arguments}

In this section there appears the function $\sqrt{x^2-1}$. There are two versions of this function:
\begin{gather*}
y_1:={\rm i}\sqrt{1-x^2},\qquad \text{for}\quad x\in D_1,
\\
y_2:=x\sqrt{1-x^{-2}},\qquad \text{for}\quad x\in\C\setminus[-1,1],
\end{gather*}
where the root denotes its principal value.
If $\pm\mathop{\rm Im} x>0$ then $y_1=\pm y_2$. Each entry in Olbricht's list has two versions depending on whether $\sqrt{x^2-1}$
has the meaning $y_k$, $k=1,2$.
We will denote
these functions by $\L{{\rm III}}{j,k}$, where $1\le j\le 24$, $1\le k\le 2$.

In Entries 1, 2, 3, 4 of Olbricht's list the hypergeometric function has the argument
\begin{gather*}
w_{17}=\frac{\sqrt{x^2-1}+x}{2\sqrt{x^2-1}} .
\end{gather*}
If we use $y_1$ for $\sqrt{x^2-1}$ then the function
\begin{gather*}
w_{17,1}=\frac{y_1+x}{2y_1}
\end{gather*}
is a conformal map from $D_1$ to the complex plane cut along the rays $(-\infty,0]$ and $[1,\infty)$.
To see this let $x=\cos\theta$. This is a conformal map from the strip $S=\{\theta\colon \mathop{\rm Re} \theta\in(0,\pi)\}$ onto $D_1$.
Then
\begin{gather*}
\frac{y_1+x}{2y_1} =\frac12 -\frac{\rm i}{2}\cot\theta.
\end{gather*}
Now $\cot\theta$ is a conformal map from $S$ to $\C\setminus((- {\rm i}\infty,-{\rm i}]\cup [{\rm i}, {\rm i}\infty))$.
Therefore, using the principal value of the hypergeometric function, the function ${}_2F_1(a,b;c;w_{17,1})$ the function is analytic on~$D_1$.
On the other hand, if we use $y_2$ for $\sqrt{x^2-1}$ then the values of
\begin{gather*}
w_{17,2}=\frac{y_2+x}{2y_2}
\end{gather*}
lie on the ray $[1,\infty)$ for $x>1$, so
${}_2F_1(a,b;c;w_{17,2})$ cannot be defined
for $x>1$ when using the principal value
of the hypergeometric function. Therefore we will not
allow $y_2$ in Entries $1,2,3,4$ of Olbrichts's list.

In Entries 5, 6, 7, 8 the hypergeometric function has argument
\begin{gather*}
w_{13}=\frac{\sqrt{x^2-1}-x}{2\sqrt{x^2-1}} .
\end{gather*}
Since $y_1$ is an even function of $x$ ($y_2$ is odd) we have
\begin{gather*}
w_{13,1}(x)=\frac{y_1-x}{2y_1}=w_{17,1}(-x)
\end{gather*}
so we can define ${}_2F_1(a,b;c;w_{13,1})$ again on $D_1$.
However, the situation is now different for $y_2$.
If we use $y_2$ for $\sqrt{x^2-1}$ then the function
\begin{gather*}
w_{13,2}=\frac{y_2-x}{2y_2}
\end{gather*}
is analytic on $D_2$ and its range is $\big\{z\in\C\colon \mathop{\rm Re} z<\frac12\big\}\setminus\big({-}\infty,-\frac12\big)$.
Therefore,
${}_2F_1(a,b;c;w_{13,2})$
is analytic on $D_2$.
This follows as in the proof of Theorem~\ref{tIII1}.

In the following list of 24 solutions of Olbricht the factors in front of the hypergeometric function are
rewritten. This is necessary in order to obtain reasonable branch cuts.
For example, solution 7 of Olbricht is
\begin{gather*}
\left(\frac{\sqrt{x^2-1}-x}{2\sqrt{x^2-1}}\right)^{\frac12\nu+\frac12} \left(\frac{\sqrt{x^2-1}+x}{2\sqrt{x^2-1}}\right)^{-\frac12\nu}{}_2F_1\left(\frac12-\mu,\mu+\frac12;\nu+\frac32;
\frac{\sqrt{x^2-1}-x}{2\sqrt{x^2-1}}\right).
\end{gather*}
If we choose $y_2$ for $\sqrt{x^2-1}$ then, up to a constant factor, this solution should agree with the
associated Legendre function $Q_\nu^\mu(x)$. However, the basis of the first power is a negative number for $x>1$. Using that
\begin{gather*}
\Big(x-\sqrt{x^2-1}\Big)\Big(x+\sqrt{x^2-1}\Big)=1,
\end{gather*}
we rewrite the solution as
\begin{gather*}
(2y_2)^{-\frac12}(x+y_2)^{-\nu-\frac12} {}_2F_1\left(\frac12-\mu,\mu+\frac12;\nu+\frac32;
\frac{y_2-x}{2y_2}\right).
\end{gather*}
This function is now analytic on $D_2$ as it should be.

The first entry in Olbricht's list is
\begin{gather*}
\L{{\rm III}}{1,1}(x):=(2y_1)^\nu {}_2F_1\left(\mu-\nu,-\mu-\nu;\frac12-\nu;\frac{y_1+x}{2y_1}\right)=\L{{\rm III}}{5,1}(-x).
\end{gather*}
The function $\L{{\rm III}}{5,1}$ will be identified below.
Also,
\begin{gather*}
\L{{\rm III}}{2,1}(x):=(2y_1)^{-\frac12}(y_1-x)^{\nu+\frac12} {}_2F_1\left(\frac12-\mu,\mu+\frac12;\frac12-\nu;\frac{y_1+x}{2y_1}\right),
\end{gather*}
and by \eqref{Euler}, $\L{{\rm III}}{2,1}=\L{{\rm III}}{1,1}$.
Also,
\begin{gather*}
\L{{\rm III}}{3,1}(x):=(2y_1)^{-\frac12}(y_1-x)^{-\nu-\frac12} {}_2F_1\left(\frac12-\mu,\mu+\frac12;\nu+\frac32;\frac{y_1+x}{2y_1}\right),
\end{gather*}
and $\L{{\rm III}}{3,1}$ is obtained from $\L{{\rm III}}{2,1}$ by replacing $\nu\mapsto -\nu-1$. Also,{\samepage
\begin{gather*}
\L{{\rm III}}{4,1}(x):=(2y_1)^{-\nu-1} {}_2F_1\left(\nu-\mu+1,\nu+\mu+1;\nu+\frac32;\frac{y_1+x}{2y_1}\right),
\end{gather*}
and by \eqref{Euler}, $\L{{\rm III}}{4,1}=\L{{\rm III}}{3,1}$.}

The fifth entry in Olbricht's list is
\begin{gather*}
\L{{\rm III}}{5,k}(x):=(2y_k)^\nu {}_2F_1\left(\mu-\nu,-\mu-\nu;\frac12-\nu;\frac{y_k-x}{2y_k}\right),
\end{gather*}
and this function is analytic on $D_k$ for $k=1,2$.
By~\cite[Section~4.1.2, Entry 31]{MOS} and \eqref{Euler},
\begin{gather*}
 \L{{\rm III}}{5,2}(x)=\pi^{-1/2}\Gamma\left(\frac12-\nu\right) \QQ_{-\nu-1}^\mu(x)\qquad\text{for}\quad x\in D_2.
\end{gather*}
Since $y_1=y_2$ for $\mathop{\rm Im} x>0$,
\begin{gather*}
 \L{{\rm III}}{5,1}(x)=\pi^{-1/2}\Gamma\left(\frac12-\nu\right) \QQ_{-\nu-1}^\mu(x)\qquad\text{if}\quad \mathop{\rm Im} x>0.
\end{gather*}
We know from~\cite[p.~147, equation~(21)]{Schafke63} that
\begin{gather*}
 \QQ_{-\nu-1}^\mu\big(1+(x-1)\expe^{2\pi {\rm i}}\big)=\expe^{-\pi {\rm i}\mu}\QQ_{-\nu-1}^\mu(x)-
 \frac{\pi {\rm i}}{\Gamma(-\nu-\mu)}P_{-\nu-1}^{-\mu}(x) .
\end{gather*}
Therefore,
\begin{gather*}
\L{{\rm III}}{5,1}(x)=\pi^{-1/2}\Gamma\left(\frac12-\nu\right)\left(\expe^{-\pi {\rm i}\mu}\QQ_{-\nu-1}^\mu(x)-\frac{\pi {\rm i}}{\Gamma(-\nu-\mu)}P_\nu^{-\mu}(x)
\right)\qquad\text{if}\quad \mathop{\rm Im} x<0.
\end{gather*}
Entries 6, 7, 8 in Olbricht's list are
\begin{gather*}
\L{{\rm III}}{6,k}(x):=(2y_k)^{-\frac12}(x+y_k)^{\nu+\frac12}
{}_2F_1\left(\frac12-\mu,\mu+\frac12;\frac12-\nu;\frac{y_k-x}{2y_k}\right),
\\
\L{{\rm III}}{7,k}(x):=(2y_k)^{-\frac12}(x+y_k)^{-\nu-\frac12}
{}_2F_1\left(\frac12-\mu,\mu+\frac12;\nu+\frac32;\frac{y_k-x}{2y_k}\right),
\\
\L{{\rm III}}{8,k}(x):=(2y_k)^{-\nu-1}
{}_2F_1\left(\nu-\mu+1,\nu+\mu+1;\nu+\frac32;\frac{y_k-x}{2y_k}\right).
\end{gather*}
By \eqref{Euler}, $\L{{\rm III}}{6,k}=\L{{\rm III}}{5,k}$ and $\L{{\rm III}}{8,k}=\L{{\rm III}}{7,k}$.
$\L{{\rm III}}{7,k}$ is obtained from $\L{{\rm III}}{6,k}$ by replacing $\nu\mapsto -\nu-1$.

In Entries 9, 10, 11, 12 the hypergeometric function has the argument
\begin{gather*}
w_{15}=\frac{2\sqrt{x^2-1}}{\sqrt{x^2-1}+x} .
\end{gather*}
If $\sqrt{x^2-1}$ is replaced by $y_1$ or $y_2$ then purely imaginary $x$ are mapped to the branch cut $[1,\infty)$ of the hypergeometric function. Therefore, we will work on $D_k^+$.
The function
\begin{gather*}
\L{{\rm III}}{9,k}(x):=(2y_k)^\mu(x+y_k)^{\nu-\mu} {}_2F_1\left(\mu-\nu,\mu+\frac12;1+2\mu;\frac{2y_k}{y_k+x}\right)
\end{gather*}
is analytic on $D_k^+$ for $k=1,2$.
By~\cite[Section~4.1.2, Entry 15]{MOS},
\begin{gather*}
 \L{{\rm III}}{9,2}(x)=4^\mu\Gamma(1+\mu) P_\nu^{-\mu}(x)\qquad\text{for}\quad x\in D_2^+.
\end{gather*}
This implies
\begin{gather*}
 \L{{\rm III}}{9,1}(x)={\rm e}^{\frac12{\rm i}\pi\mu}4^\mu\Gamma(1+\mu) \P_\nu^{-\mu}(x)\qquad\text{for}\quad x\in D_1^+.
\end{gather*}
Also,
\begin{gather*}
\L{{\rm III}}{10,k}(x):= (2y_k)^{-\mu}(x+y_k)^{\mu+\nu}{}_2F_1\left({-}\nu-\mu,\frac12-\mu;1-2\mu;\frac{2y_k}{y_k+x}\right)
\end{gather*}
so $\L{{\rm III}}{10,k}$ is obtained from $\L{{\rm III}}{9,k}$ by replacing $\mu\mapsto -\mu$. Hence
\begin{gather*}
\L{{\rm III}}{10,2}(x)=4^{-\mu}\Gamma(1-\mu) P_\nu^\mu(x)\qquad\text{for}\quad x\in D_2^+.
\end{gather*}
Also,
\begin{gather*}
\L{{\rm III}}{11,k}(x):=(2y_k)^{-\mu}(y_k+x)^{-1+\mu-\nu} {}_2F_1\left(\nu-\mu+1,\frac12-\mu;1-2\mu;\frac{2y_k}{y_k+x}\right),
\end{gather*}
and $\L{{\rm III}}{11,k}$ is obtained from $\L{{\rm III}}{10,k}$ by replacing $\nu\mapsto -\nu-1$. By \eqref{Euler},
$\L{{\rm III}}{11,k}=\L{{\rm III}}{10,k}$ on $D_k^+$. Also,
{\samepage\begin{gather*}
\L{{\rm III}}{12,k}(x):=(2y_k)^\mu(y_k+x)^{-1-\mu-\nu} {}_2F_1\left(\nu+\mu+1,\mu+\frac12;1+2\mu;\frac{2y_k}{y_k+x}\right) .
\end{gather*}
$\L{{\rm III}}{12,k}$ is obtained from $\L{{\rm III}}{9,k}$ by replacing $\nu\mapsto -\nu-1$. By \eqref{Euler},
$\L{{\rm III}}{12,k}=\L{{\rm III}}{9,k}$ on $D_k^+$.}

Entries 13, 14, 15, 16 in Olbricht's list use the argument
\begin{gather*}
w_{16}=\frac{2\sqrt{x^2-1}}{\sqrt{x^2-1}-x} .
\end{gather*}
Again purely imaginary $x$ are mapped to the branch cut $[1,\infty)$ of the hypergeometric function.
Olbricht has
\begin{gather*}
\L{{\rm III}}{13,k}(x):=(2y_k)^\mu(x-y_k)^{\nu-\mu} {}_2F_1\left(\mu-\nu,\mu+\frac12;1+2\mu;\frac{2y_k}{y_k-x}\right),
\\
\L{{\rm III}}{14,k}(x):=(2y_k)^{-\mu}(x-y_k)^{\mu+\nu} {}_2F_1\left({-}\nu-\mu,\frac12-\mu;1-2\mu;\frac{2y_k}{y_k-x}\right),
\\
\L{{\rm III}}{15}(x):=(2y_k)^{-\mu}(x-y_k)^{-1+\mu-\nu} {}_2F_1\left(\nu-\mu+1,\frac12-\mu;1-2\mu;\frac{2y_k}{y_k-x}\right),
\\
\L{{\rm III}}{16}(x):=(2y_k)^\mu(x-y_k)^{-1-\mu-\nu} {}_2F_1\left(\nu+\mu+1,\mu+\frac12;1+2\mu;\frac{2y_k}{y_k-x}\right).
\end{gather*}
Using \eqref{Pfaff1} these functions can be reduced to Entries 9, 10, 11, 12 as follows
\begin{gather*}
\L{{\rm III}}{j,k}=\L{{\rm III}}{j-4,k} \qquad \text{for}\quad j=13,14,15,16.
\end{gather*}
Entries 17 to 24 in Olbricht's list are
\begin{gather*}\allowdisplaybreaks
\L{{\rm III}}{17,k}(x):=(2y_k)^\mu(y_k-x)^{\nu-\mu} {}_2F_1\left(\mu-\nu,\mu+\frac12;\frac12-\nu;\frac{x+y_k}{x-y_k}\right),
\\
\L{{\rm III}}{18,k}(x):=(2y_k)^{-\mu}(y_k-x)^{\mu+\nu} {}_2F_1\left({-}\nu-\mu,\frac12-\mu;\frac12-\nu;\frac{x+y_k}{x-y_k}\right),
\\
\L{{\rm III}}{19,k}(x):=(2y_k)^\mu(y_k-x)^{-\nu-\mu-1}{}_2F_1\left(\frac12+\mu,\nu+\mu+1;\nu+\frac32;\frac{x+y_k}{x-y_k}\right),
\\
\L{{\rm III}}{20,k}(x):=(2y_k)^{-\mu}(y_k-x)^{\mu-\nu-1}{}_2F_1\left(\frac12-\mu,\nu-\mu+1;\nu+\frac32;\frac{x+y_k}{x-y_k}\right),
\\
\L{{\rm III}}{21,k}(x):=(2y_k)^\mu(x+y_k)^{\nu-\mu} {}_2F_1\left(\mu-\nu,\mu+\frac12;\frac12-\nu;\frac{x-y_k}{x+y_k}\right),
\\
\L{{\rm III}}{22,k}(x):=(2y_k)^{-\mu}(x+y_k)^{\mu+\nu} {}_2F_1\left({-}\nu-\mu,\frac12-\mu;\frac12-\nu;\frac{x-y_k}{x+y_k}\right),
\\
\L{{\rm III}}{23,k}(x):=(2y_k)^\mu(x+y_k)^{-1-\mu-\nu}{}_2F_1\left(\frac12+\mu,\nu+\mu+1;\nu+\frac32;\frac{x-y_k}{x+y_k}\right),
\\
\L{{\rm III}}{24,k}(x):=(2y_k)^{-\mu}(x+y_k)^{-1+\mu-\nu}{}_2F_1\left(\frac12-\mu,\nu-\mu+1;\nu+\frac32; \frac{x-y_k}{x+y_k}\right).
\end{gather*}
Using \eqref{Pfaff1}, these functions can be reduced to Entries 1 to 8 as follows
\begin{gather*}
\L{{\rm III}}{j,k}=\L{{\rm III}}{j-16,k} \qquad \text{for}\quad j=17,18,\dots,24.
\end{gather*}

\section[Values on the cut [1,infty) of the Gauss hypergeometric function]
{Values on the cut $\boldsymbol{[1,\infty)}$ of the Gauss \\hypergeometric function}\label{hypercutsection}

The hypergeometric function has a branch cut along the ray $[1,\infty)$. Along this branch cut it has limiting values
$\F{a,b}{c}{x\pm {\rm i}0}$, $x>1$, that are usually different when we approach the branch cut from above or below.
Since these limiting values play a role in this paper, we note the following results.

\begin{thm}\label{Theorem1}
Let $a,b\in\C$, $c\in\C\setminus-\N_0$, $x\in(1,\infty)$. Then
\begin{gather*}
\F{a,b}{c}{x\pm {\rm i}0}=\frac{\Ga(c)\Ga(b-a)}{\Ga(b)\Ga(c-a)}\expe^{\pm {\rm i}\pi a}x^{-a}
\F{a,a-c+1}{a-b+1}{\frac1x}\nonumber
\\ \hphantom{\F{a,b}{c}{x\pm {\rm i}0}=}
{}+\frac{\Ga(c)\Ga(a-b)}{\Ga(a)\Ga(c-b)}\expe^{\pm {\rm i}\pi b}x^{-b}
\F{b,b-c+1}{b-a+1}{\frac1x}\!.
\end{gather*}
\end{thm}
\begin{proof}
The connection relation~\cite[equation~(15.8.2)]{NIST:DLMF} and \eqref{Eulerreflect} give
\begin{gather*}
\F{a,b}{c}{z}=\frac{\Ga(c)\Ga(b-a)}{\Ga(b)\Ga(c-a)}
(-z)^{-a}\F{a,a-c+1}{a-b+1}{\frac1z}
\\ \hphantom{\F{a,b}{c}{z}=}
{}+\frac{\Ga(c)\Ga(a-b)}{\Ga(a)\Ga(c-b)}
(-z)^{-b}\F{b,b-c+1}{b-a+1}{\frac1z}\!.
\end{gather*}
This identity holds for all $z\in\C\setminus[0,\infty)$.
Setting $z=x+i\epsilon$, $x>1$, $\epsilon\in\R\setminus\{0\}$, and letting $\epsilon\to0^+$ and $\epsilon\to0^-$,
we obtain the stated result.
\end{proof}

We note that the limiting values of the generalized hypergeometric function on the branch cut have recently been investigated by Karp and Prilepkina~\cite{KarpPrilepkina17}.
Using~\cite[equations~(15.8.3), (15.8.4) and~(15.8.5)]{NIST:DLMF}, one may also obtain alternative
representations of the limiting values of~the hypergeometric function on the branch cut.

\begin{thm}
Let $a,b\in\C$, $c\in\C\setminus-\N_0$, $x\in(1,\infty)$. Then
\begin{gather*}
\F{a,b}{c}{x\pm {\rm i}0}= \frac{
\Ga(c)\Ga(c-a-b)}{\Ga(c-a)\Ga(c-b)}
x^{1-c}\F{a-c+1,b-c+1}{a+b-c+1}{1-x}
\\ \hphantom{\F{a,b}{c}{x\pm {\rm i}0}=}
{}+\expe^{\pm {\rm i}\pi(a+b-c)}\frac{\Ga(c)\Ga(a\!+b\!-c)}{\Ga(a)\Ga(b)}
(x\!-1)^{c-a-b}\F{c-a,c-b}{c\!-a\!-b\!+1}{1\!-x}\!.
\end{gather*}
\end{thm}

\begin{thm}
Let $x\in(1,\infty)$. Then
\begin{gather*}
\F{a,b}{c}{x\pm {\rm i}0}=
\frac{\Ga(c)\Ga(c-a-b)}{\Ga(c-a)\Ga(c-b)}
x^{-a}\F{a,a-c+1}{a+b-c+1}{1-\frac{1}{x}}
\\ \hphantom{\F{a,b}{c}{x\pm {\rm i}0}=}
{}+\expe^{\pm {\rm i}\pi(a+b-c)}\frac{\Ga(c)\Ga(a+b-c)}{\Ga(a)\Ga(b)}
(x-1)^{c-a-b}x^{a-c}
\\\hphantom{\F{a,b}{c}{x\pm {\rm i}0}=+}
{}\times\F{1-a,c-a}{c-a-b+1}{1-\frac{1}{x}}\!.
\end{gather*}
\end{thm}

\begin{thm}\label{Theorem4}
Let $x\in(1,\infty)$. Then
\begin{gather*}
\F{a,b}{c}{x\pm {\rm i}0}=
\expe^{\pm {\rm i}\pi a}\frac{\Ga(c)\Ga(b-a)}{\Ga(b)\Ga(c-a)}
(x-1)^{-a}\F{a,c-b}{a-b+1}{\frac1{1-x}}
\\ \hphantom{\F{a,b}{c}{x\pm {\rm i}0}=}
{}+\expe^{\pm {\rm i}\pi b}\frac{\Ga(c)\Ga(a-b)}{\Ga(a)\Ga(c-b)}
(x-1)^b\F{b,c-a}{b-a+1}{\frac1{1-x}}\!.
\end{gather*}
\end{thm}

\pdfbookmark[1]{References}{ref}
\LastPageEnding

\end{document}